\documentclass[10pt,reqno]{amsart}

\usepackage{graphicx}
\usepackage{tabularx}
\usepackage{xargs}
\usepackage{geometry}
\usepackage{fancyhdr} 
\usepackage{lastpage} 
\usepackage{extramarks} 
\usepackage[most]{tcolorbox} 
\usepackage{xcolor} 
\usepackage[hidelinks]{hyperref} 
\usepackage[permil]{overpic}
\usepackage{pict2e} 
\usepackage{listings}
\usepackage{amssymb}
\usepackage{amsthm}
\theoremstyle{plain}
\usepackage[
backend=biber,
style=numeric,
sorting=ynt
]{biblatex}
\usepackage{xargs}
\usepackage{float}
\usepackage{lipsum}
\usepackage{geometry}
\usepackage{fancyhdr} 
\usepackage{lastpage} 
\usepackage{extramarks} 
\usepackage[most]{tcolorbox} 

\usepackage{xcolor} 
\usepackage[hidelinks]{hyperref} 
\usepackage[permil]{overpic}
\usepackage{pict2e} 
\usepackage{listings}

\newtheorem*{theorem*}{Theorem}

\numberwithin{equation}{section}
\theoremstyle{plain}
\newtheorem{definition}{Definition}[section]
\newtheorem{theorem}[equation]{Theorem}

\newtheorem{remark}[equation]{Remark}
\newtheorem{corollary}[equation]{Corollary}
\newtheorem{lemma}[equation]{Lemma}

\newtheorem{proposition}[equation]{Proposition} 
\providecommand{\keywords}[1]
{
  \small	
  \textbf{\textit{Keywords---}} #1
}

\title{Unbounded dynamics for three dimensional vector fields}
\author{Eran Igra$^{1,2}$}
\address{1. Shanghai Institute for Mathematics and Interdisciplinary Sciences (SIMIS), Shanghai 200433, China\newline
\newline
2. Research Institute of Intelligent Complex Systems, Fudan University, Shanghai 200433, China}
\email{eranigra@simis.cn}

\begin{document}

\begin{abstract}
Given a three-dimensional vector field $F$ with a finite number of fixed points - what can we say on its unbounded dynamics? We tackle this question, and prove sufficient conditions for $F$ to have fixed points with unbounded invariant manifolds. Following that, we use these results to study the dynamics of the Genesio--Tesi system, the Belousov--Zhabotinsky reaction, and the Michelson system.
\end{abstract}

\maketitle

\section{Introduction}

Let $\dot{s}=F(s)$, $s=(x,y,z)$ be a smooth vector field of $\mathbb{R}^3$ which continuously extends to $S^3$, with $\infty$ added as a fixed point. Moreover, assume $F$ has a finite number of fixed points in $\mathbb{R}^3$. The question which we tackle in this paper is the following - what is the connection between the bounded and unbounded dynamics of the flow? Or, more precisely, what conditions should $F$ satisfy such that the flow always mixes a neighborhood of $\infty$ and a neighborhood of the fixed points? This question is particularly interesting when the sign of the divergence of $F$ is not constant in $\mathbb{R}^3$ - or put simply, in scenarios when one should not expect $F$ to generate a global attracting (or repelling) invariant set for the flow. \\

Needless to say, the problem of describing the unbounded dynamics of vector fields is not a new one. The most famous approach to study it was originally considered by H. Poincaré - who devised the notion of the Poincaré sphere or Poincaré compactification to do so (see Ch.3.10 in \cite{Perko} and Section 2 in \cite{LiLl} for a survey of these ideas). Informally speaking, given a polynomial vector field $F$ of $\mathbb{R}^3$ the Poincaré sphere method is a way to describe the local dynamics at $\infty$ by expanding $\infty$ to a direction sphere identified with $S^2$, and extending the flow to it. And indeed, since its introduction the Poincaré sphere method was applied extensively to study the unbounded dynamics of many vector fields - see, for example, \cite{LiLl}, \cite{HYCX}, \cite{Yun}, \cite{Mes} and \cite{Mor} among many others.\\

And yet, while teaching us a great deal about the local dynamics at $\infty$ the Poincaré sphere method often cannot teach us too much about the global topological dynamics of the flow. In particular, it often does not teach us how the local dynamics around $\infty$ and those around the fixed points are connected (if at all). It is precisely this gap we address in this paper. Namely, we prove the following Theorem:

\begin{theorem}
\label{in}    Let $\dot{s}=F(s)$, $s=(x,y,z),F=(F_1,F_2,F_3)$ be a $C^k$, $k>0$ vector field satisfying the following:
    \begin{enumerate}
        \item $F$ extends continuously to $\infty$, where $\infty$ is added as a fixed point whose Poincaré index of $F$ at $\infty$ is either $0$ or $\pm1$ (see Definition \ref{index}).
        \item $F$ has finitely many fixed points in $\mathbb{R}^3$, all non-degenerate, at which the linearization has $x_1,\ldots ,x_n$ complex-conjugate eigenvalues.
        \item There exists some $i\in\{1,2,3\}$ such that the level set $H=\{s\in\mathbb{R}^3\mid F_i(s)=0\}$ satisfies the following:
        \begin{itemize}
             \item $H$ is unbounded, and homeomorphic to a plane.
            \item Let $l$ denote the tangency set of the vector field $F$ to $H$. Then, $l$ is homeomorphic to $\mathbb{R}$.
            \item Parameterizing $l=(s_1(t),s_2(t),s_3(t))$, $t\in\mathbb{R}$, then $s_i(t)$ is an increasing function in $t$.
            \item For all $c\in\mathbb{R}$, $H$ is transverse to the plane $H_c=\{s\in\mathbb{R}^3\mid s_i=c\}$ - and in particular, $l_c=H_c\cap H$ is unbounded and homeomorphic to $\mathbb{R}$. Moreover, $l_c\cap l$ is a singleton.
            \item Let $I$ be a segment on $l$ such that there are no fixed points on $I$. Then, for all $s\in I$ there exists some $\epsilon>0$ depending on $s$ such that the surface $\cup_{s\in I}\phi_{(-\epsilon,\epsilon)}(s)$ is either in $\{F_i(s)\geq0\}$ or in $\{F_i(s)\leq0\}$.
            \item Finally, $H\setminus l$ is composed of two (topological) half-planes, $H_+$ and $H_-$ such that on $H_+$ initial conditions cross from $\{F_i(s)\leq0\}$ into $\{F_i(s)>0\}$, and on $H_-$ from $\{F_i(s)\geq0\}$ to $\{F_i(s)<0\}$.
        \end{itemize}
    \end{enumerate}

Then, there exist two fixed points $x_1,x_2$ (not necessarily distinct) such that each fixed point generates a respective one-dimensional invariant manifold, $\Gamma_1$ and $\Gamma_2$, connecting it to $\infty$ (see Fig.\ref{conn}). Moreover, there exists a curve $\gamma\subseteq\mathbb{R}^3$ connecting $x_1$ and $x_2$ such that the union $\{x_1,x_2,\infty\}\cup\Gamma_1\cup\Gamma_2\cup\gamma$ is a knot in $S^3$, ambient isotopic to $S^1$, the unknot. In addition, whenever $F$ generates more than one-fixed point we have $x_1\ne x_2$.
\end{theorem}

\begin{figure}[h]
\centering
\begin{overpic}[width=0.35\textwidth]{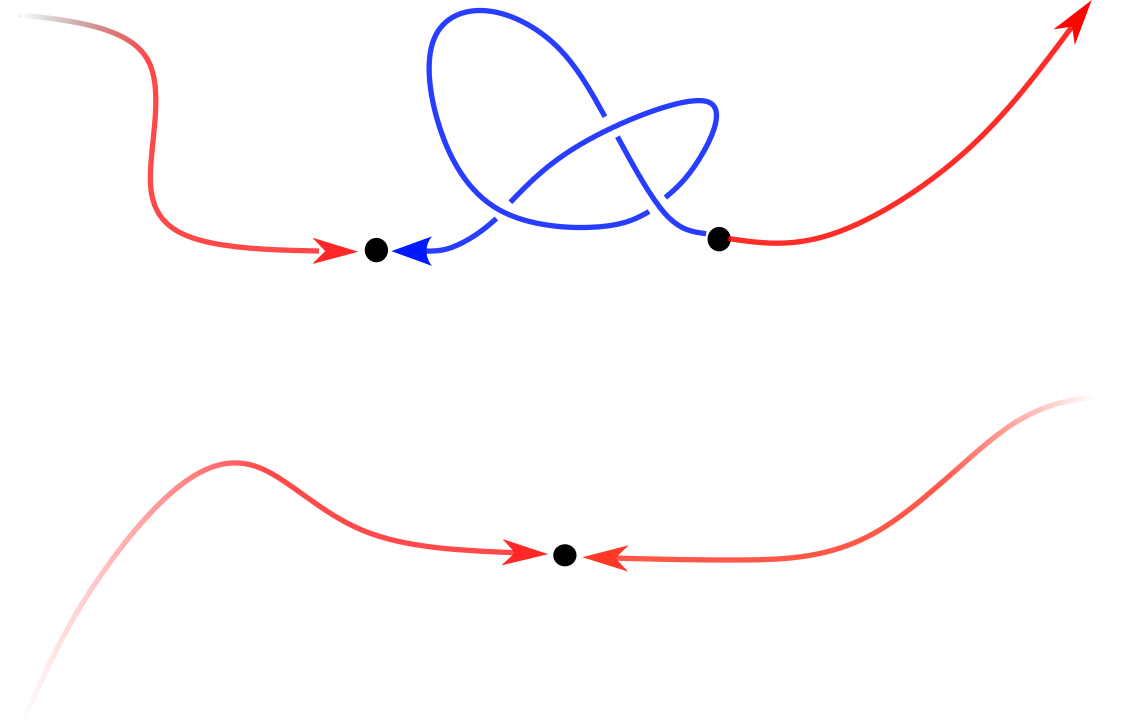}
\put(340,215){$\Gamma_2$}
\put(720,250){$\Gamma_1$}
\put(600,390){$x_2$}
\put(720,500){$\Gamma_2$}
\put(480,200){$x$}
\put(160,500){$\Gamma_1$}
\put(350,390){$x_1$}
\end{overpic}
\caption{\textit{The fixed points $x_1$ and $x_2$ and their unbounded invariant manifolds, sketched in red (in the upper scenario, there also exists a bounded heteroclinic trajectory sketched in blue). In the lower scenario, we have $x_1=x_2=x$.}}\label{conn}

\end{figure}

Despite its technical formulation, Theorem \ref{in} has the following heuristic meaning: let $F$ be a vector field satisfying the assumptions of Theorem \ref{in} and consider two small neighborhoods $B_F$ of the Fixed Points and $B_\infty$ of $\infty$ (with the latter taken in $S^3$) - then, the flow generated by $F$  mixes the trajectories of $B_F$ and $B_\infty$. More precisely, Theorem \ref{in} gives us a sufficient conditions which force $F$ to carry the trajectories of initial conditions from either $B_F$ to $B_\infty$, $B_\infty$ to $B_F$, or both. As such, Theorem \ref{in} teaches us how the local dynamics around the fixed points interact with the local dynamics around $\infty$. Before moving on, we further remark that despite the seemingly long list of assumptions, Theorem \ref{in} is in fact very general - in the sense that it can be applied to study a relatively large class of three dimensional flows in $\mathbb{R}^3$, as exemplified in Section $2$.\\

This paper is organized as follows - in Section $2$ we define the notions and ideas used in this paper, after which we prove Theorem \ref{in} by direct qualitative (and topological) analysis of the vector field $F$. Following that, in Sections \ref{ap1} and \ref{ap2} we apply Theorem \ref{in} to derive several results on the Belousov--Zhabotinsky Reaction, the Genesio--Tesi system, and the Michelson system (see \cite{BZ}, \cite{GT} and \cite{Michh} respectively). We conclude this paper in Section \ref{nonstand}, where we show how one can modify the arguments used to prove Theorem \ref{in} in to study special cases of vector fields which do not satisfy all the assumptions of Theorem \ref{in}. We exemplify that by studying the dynamics of a vector field originally introduced in \cite{Mor}.\\

Finally, before we begin we would like to remark that the proof of Theorem \ref{in} is to a large degree inspired by Proposition 2.3 in \cite{Pi} and by the author's own work in \cite{I} on the Rössler system (see Theorem 2.8). In fact, Theorem \ref{in} originated in an attempt to prove this result forms an instance of a larger theory.

\subsection*{Acknowledgements}
The author would like to thank Konstantin Khanin for his suggestion of this project, and to Noy Soffer-Aranov and Valerii Sopin for their helpful comments and suggestions.

\section{The general theorem:}

In this section we prove Theorem \ref{in}. From now on $F=(F_1,F_2,F_3)$ will always denote a smooth vector field of $\mathbb{R}^3$ and $s\in\mathbb{R}^3$ will always denote a shorthand notation for an initial condition $(x,y,z)$. As mentioned above, we prove Theorem \ref{in} by using direct qualitative analysis of the vector field $F$, and we are particularly interested in the case where $F$ is a polynomial vector field - i.e., each component $F_i$, $i=1,2,3$ is a multivariate polynomial in $x,y,z$. Whenever this is the case, we refer to $F$ as a \textbf{polynomial vector field}. With these ideas in mind, we first prove the following fact about polynomial vector fields - which, even though we will not use directly, will motivate our results later on in this paper:
\begin{proposition}
\label{polyinf} With previous notations    Let $\dot{s}=F(s)$, $F=(F_1,F_2,F_3)$ be a polynomial vector field of $\mathbb{R}^3$ with a finite number of fixed points. Then, whenever $F$ has a finite number of fixed points is $\mathbb{R}^3$, $F$ can be modified around $\infty$ s.t. it extends continuously to $S^3$ with $\infty$ added as a fixed point for the flow. Moreover, this extension is orbitally equivalent to the flow of the original vector field in a punctured neighborhood of $\infty$.
\end{proposition}
\begin{proof}
If $F$ is a linear vector field, the proof is immediate and follows easily by proving that if a matrix is invertible, then whenever $\mid \mid s_n\mid \mid \to\infty$ we have $\mid \mid F(s_n)\mid \mid \to\infty$. Therefore, assume now that $F$ is non-linear - i.e., every $F_i$, $i=1,2,3$ is a multivariate polynomial of degree at most $n$ (for some $n>1$). From now on, we further assume $n>1$ is minimal - i.e., we assume that $n>1$ is such that  exists at least one $i\in\{1,2,3\}$ such that $F_i$ is a multivariate polynomial in $s=(x,y,z)$ of degree exactly $n$. Moreover, setting $\bullet$ as the inner product, observe $F(s)\bullet s=F_1(s)\bullet x+F_2(s)\bullet y+F_3(s)\bullet z$ is also a multivariate polynomial - in particular, we can write $F(s)\bullet s=P(s)$, where $P(s)$ is a multivariable polynomial of degree $n+1$. Now, convert the above equation to spherical coordinates $(r,\theta,\varphi)$ - where $r=\mid \mid (x,y,z)\mid \mid $, $\theta\in[0,\pi]$ and $\varphi\in[0,2\pi)$. By the discussion above we can rewrite the above equation as $F(s)\bullet s=\sum_{i=0}^{n+1}a_i r^{i} g_i(\theta,\varphi)$ - where $(a_0,\ldots ,a_{n+1})$ is a non-zero vector and $g_0,\ldots ,g_{n+1}$ are the products and sums of trigonometric functions in $\theta$ and $\varphi$. In particular, the functions $g_0,\ldots ,g_{n+1}$ are independent of $r$ (some of them are possibly constants).\\

To continue, consider any $r>0$ sufficiently large such that there are no fixed points on the sphere $S_r=\{(x,y,z)\mid \mid \mid (x,y,z)\mid \mid =r\}$ - we now claim there exists some maximal $k$, $n+1\geq k\geq0$, such that $g_k(\theta,\varphi)$ is not identically zero on $S_r$. To see why, assume this is not the case - i.e., assume that for all $n+1\geq k>0$ the function $g_k$ vanishes identically - which implies $F(s)\bullet s=0$ on $S_r$, i.e., $F$ is tangent to $S_r$. This implies the vector field $\frac{F(s)}{\mid \mid F(s)\mid \mid }, s\in S_r$ is tangent to the sphere $S^2$ - and since $r$ was chosen such that $F$ has no fixed points on $S_r$ we conclude $\frac{F(s)}{\mid \mid F(s)\mid \mid }$ has no fixed points on $S^2$ as well. In other words, $\frac{F(s)}{\mid \mid F(s)\mid \mid }$ is a continuous, non-vanishing vector field of $S^2$ - which contradicts the Hairy Ball Theorem. As a consequence, we conclude there exists at least one $k>0$ such that $g_k$ is not identically $0$ - from now on, we assume $n+1\geq k>0$ is maximal with respect to this property.\\

Now, consider $(\theta,\varphi)\in[0,\pi]\times[0,2\pi)$ such that $g_k(\theta,\phi)\ne0$. The above decomposition of $F(s)\bullet s$ implies that $\lim_{r\to\infty}\frac{F(s)\bullet s}{r^{k}g_k(\theta,\varphi)}=1$. We now claim $g_k$ is non-zero in an open and dense subset of $I=[0,\pi]\times[0,2\pi)$ - which will immediately imply $\lim_{r\to\infty}\frac{F(s)\bullet s}{r^{n+1}g_k(\theta,\varphi)}=1$ in an open and dense collection of $(\theta,\varphi)$. We do so by contradiction - to this end, note the set $D=\{(\theta,\varphi)\in I\mid g_k(\theta,\varphi)\ne0\}$ is trivially open, we now prove it is dense (we already know it is non-empty). To this end, note we can decompose $g_k(\theta,\varphi)=\sum_{i=1}^d f_{1,i}(\theta)f_{2,i}(\varphi)$ - where $f_{i,j}$ are trigonometric functions. This implies that $f_{j,i}$ is holomorphic in either $\theta$ or $\varphi$ (respectively) - hence for all $c\in[0,2\pi)$, $d\in[0,\pi]$ the functions $g_k(c,\varphi)$ and $g_k(\theta,d)$ are also holomorphic. Therefore, by the Identity Theorem for holomorphic maps we conclude that given any $c\in[0,2\pi)$ for which there exists an open sub-interval of $[0,\pi]\times\{c\}$ on which $g_k$ vanishes, $g_k$ has to be identically $0$ throughout the line $[0,\pi]\times\{c\}$. Similarly, given any $c\in[0,\pi]$, if there exists some open sub-interval of $\{c\}\times[0,2\pi)$ on which $g_k$ vanishes, it must also vanish throughout $\{c\}\times[0,2\pi)$.\\

As a consequence, it follows that if $D$ is not dense in $I$, there exists $D'$, an open set of $[0,\pi]\times[0,2\pi)$ and some straight line $l$ which is parallel to either $[0,\pi]\times\{0\}$ or to $\{0\}\times[0,2\pi)$ such that $l$ intersects both $D$ and $D'$. By the paragraph above, this implies $g_k$ must vanish throughout $l$ - and in particular, also on $l\cap D$. Since we already know $D\ne\emptyset$ and since in addition for all $(\theta,\varphi)\in D$ we have $g_k(\theta,\varphi)\ne0$ we have a contradiction - i.e., the set $D'$ cannot include an open set, i.e., it is nowhere dense, and consequentially, $D$ is dense. As we already know $D$ is open, it follows $\overline{D}=[0,\pi]\times[0,2\pi]$ which implies that for a dense collection of $(\theta,\varphi)\in [0,2\pi]\times[0,2\pi)$ we have $g_k(\theta,\varphi)\ne0$ - hence we have $\lim_{r\to\infty}\frac{F(s)\bullet s}{r^{k}g_k(\theta,\varphi)}=1$ throughout $D$.\\

It now follows that as $\mid \mid s\mid \mid =r\to\infty$ the limiting behavior of $F(s)\bullet s$ is independent of $r$. In other words, on any sufficiently large sphere $S_r$ the inner product satisfies $F(s)\bullet \frac{s}{\mid \mid s\mid \mid }\approx g_k(\theta,\varphi)$, $s=(r,\theta,\varphi)$ - and in particular, by the continuity of $F$ and $g_k$ we have $\lim_{\mid \mid s\mid \mid =r,r\to\infty}F(s)\bullet\frac{s}{\mid \mid s\mid \mid }=g_k(\theta,\varphi)$. This implies we can add $\infty$ as a fixed point for the flow generated by $F$ - or, in other words, we can extend the vector field $F$ continuously to $\infty$ by adding $\infty$ as a fixed point. Specifically, by "blowing up" $\infty$ into the sphere $S^2$, by the discussion above we note the behavior of $g_k$ on $S^2$ is independent of the magnitude $r$, and depends only on the region of $S^2$ which intersects $(r,\theta,\varphi)$ is. In other words, by shrinking $S^2$ back to $\infty$ we get that the behavior of $F$ around $\infty$ is well-defined. In particular, on rays $\{(r,\theta,\varphi)|r>0\}$ where $g_k(\theta,\varphi)$ is positive, the flow would push towards $\infty$ - conversely, where it is negative, the flow would pull away from $\infty$. Since we can partition a punctured neighborhood of $\infty$ (in $S^3$) to regions where $g_k$ is positive, negative, and vanishes (and since $F$ is both continuous and smooth in that neighborhood), it follows these attracting and repelling directions to $\infty$ are well-defined.\\

To conclude the proof, we now proceed by choosing some sufficiently small punctured neighborhood $N$ of $F$ around $\infty$. As discussed above, we can partition this neighborhood to flow lines which are repelled away or attracted to $\infty$. By decreasing the speed along these flow lines (without changing their topology) as $r\to\infty$ in $N$ we get a modification of $F$ - a new flow, which, by its definition, is orbitally equivalent to the previous one. Moreover, as this new flow naturally extends to $\infty$ by vanishing there, it follows $\infty$ is a fixed point for the said flow. The proof of Proposition \ref{polyinf} is now complete.
\end{proof}
\begin{remark}
    Using $n-$dimensional spherical coordinates (see \cite{hyper}), one can generalize Proposition \ref{polyinf} to $n-$dimensional polynomial vector fields. 
\end{remark}

Before moving on, we remark that in general one should not assume the argument above extends the vector field $F$ smoothly to $\infty$ - or in other words, given a polynomial vector field $\dot{s}=F(s)$ which satisfies the assertions of Proposition \ref{polyinf}, there is no reason to assume $\infty$ will be a smooth fixed point for the flow - this would very much depend on the properties of $g_k$, which cannot be assumed to have "simple" behavior. This leads us to ask the following: since in general we cannot hope to study the dynamics of $F$ at the fixed point at $\infty$ using smooth tools, how can we study it? To this end, we now introduce the following definition, inspired by the Poincaré index of a smooth vector field (see Ch.6 in \cite{Mil}). We first recall that if $x\in\mathbb{R}^3$ is an isolated fixed point for a vector field $F$, the \textbf{Poincaré index} is defined as the degree of $\frac{F}{\mid \mid F\mid \mid }$ on some sufficiently small sphere centered at $x$. Using the homotopy invariance of the degree of sphere maps, we now generalize this definition as follows:
\begin{definition}
    \label{index} Let $F$ be a smooth vector field of $\mathbb{R}^3$ which has a finite number of fixed points and extends continuously to $\infty$ as in Prop.\ref{polyinf}. Let $r>0$ is sufficiently large such that all the fixed points of $F$ are inside the ball $B_r=\{(x,y,z)\mid \mid \mid (x,y,z)\mid \mid <r\}$ and set $S_r=\{(x,y,z)\mid \mid \mid (x,y,z)\mid \mid =r\}$. Then, we define the \textbf{Poincaré index of} $\infty$ as $-d$, where $d$ is the degree of $\frac{F(s)}{\mid \mid F(s)\mid \mid }$ on $S_r$.
\end{definition}

Definition \ref{index} will be instrumental to the proof of Theorem \ref{in} - and it is easy to see it generalizes the notion of the Poincaré index of a smooth fixed point for a flow (see Lemma 6.1 in \cite{Mil}). However, before we state and prove the said Theorem, for completeness, we first show Definition \ref{index} is well-defined. In other words, we first prove the following Lemma:
\begin{lemma}
\label{welldefi}   With the notations of Definition \ref{index}, the Poincaré index at $\infty$ is well-defined - that is, the Poincaré index at $\infty$ does not depend on $r>0$. Moreover, if  $-d\in\{0,1,-1\}$ the sum of the Poincaré indices of the zeros of $F$ inside $S_r$ is $d$.
\end{lemma}
\begin{proof}
    Let $r>0$ be sufficiently large as in Definition \ref{index}, let $B_r$ and $S_r$ denote the same sets as above, and set $v=\frac{F}{\mid \mid F\mid \mid }$. Whenever $F$ has no fixed points on $S_r$, $v$ maps $S_r$ smoothly to $S^2$. Now, consider any sufficiently large $r_1,r_2>0$ such that all the fixed points of $F$ are trapped inside $B_{r_1}$ and $B_{r_2}$. For all such $r_1$ and $r_2$ the maps $v:S_{r_1}\to S^2$ and $v:S_{r_2}\to S^2$ are homotopic - which implies they have the same degree as a sphere map. As the Poincaré index at $\infty$ is defined by the degree of $v$ on any sufficiently large sphere $S_r$, this proves the Poincaré index at $\infty$ is well-defined and does not depend on $r$.\\

    To conclude the proof, we need to show that when $d\in\{0,1,-1\}$ and $r>0$ is sufficiently large, the sum of the Poincaré indices of the fixed points inside $\{s\mid \mid \mid s\mid \mid <r\}$ equals $d$. To do so, choose some $r>0$ such that all the fixed points of $F$ are trapped in $B_r$ - additionally, recall that given a smooth vector field $G$ of $S^3$ with fixed points $x_1,\ldots  x_n$ with Poincaré indices $d_1,\ldots ,d_n$, by the Poincaré-Hopf Theorem we have $d_1+\ldots +d_n=0$. Since the Poincaré index of $F$ at $\infty$ is $-d$, by the Th.6.1 in \cite{Mil} it follows that any smoothening of $F$ in $\{s|||s||>r\}$ with non-degenerate zeros will add at most fixed points whose indices sum to $-d$. As such, by the Poincar\'e-Hopf Theorem we conclude that the sum of Poincar\'e Indices at $\{s|||s||<r\}$ can only be $d$, as the sums of the indices in $\{s|||s||<r\}$ and $\{s|||s||>r\}$ must cancel each other. The proof is now complete.
\end{proof}
Having proven Lemma \ref{welldefi}, we now prove its subsequent corollary:
\begin{corollary}
    \label{zeroone} Under the assumptions and notations of Lemma \ref{welldefi}, let $x_1,\ldots ,x_n$ denote the fixed points of $F$ in $\mathbb{R}^3$, and let $d_1,\ldots d_n$ denote their indices. Assume that in addition $\infty$ is added as a fixed point of Poincaré index $-d$, $d\in\{0,1,-1\}$, and that we have the equality $d=\sum_{i=1}^nd_i$. Then, given any sufficiently large $r>0$, it is possible to continuously deform $F$ in $C_r=\{(x,y,z)\mid \mid \mid (x,y,z)\mid \mid >r\}$ such that $\infty$ is deformed into a smooth fixed point for the flow of Poincaré index $0$ or a Saddle Focus.
\end{corollary}
\begin{proof}
  We first recall that by Hopf's theorem, two smooth maps $f,g:S^2\to S^2$ are homotopic if and only if they have the same degree. Now, choose any $r>0$ that is sufficiently large such that $C_r$ includes no fixed points for $F$ - this implies that for all $r'>r$, the maps $\frac{F}{\mid \mid F\mid \mid }:S_r\to S^2$, $\frac{F}{\mid \mid F\mid \mid }:S_{r'}\to S^2$ are homotopic and have the same degree. This implies that whenever the Poincaré index at $\infty$ is $0$ we can homotopically deform the dynamics inside $C_r$ such that $\infty$ becomes a smooth fixed point of Index $0$, as illustrated in Fig.\ref{ind}. Moreover, we can perform the smoothening such that $\infty$ is the only fixed point for the flow as illustrated in Fig.\ref{ind} - and hence, must be of Poincaré index $0$.\\

\begin{figure}[h]
\centering
\begin{overpic}[width=0.3\textwidth]{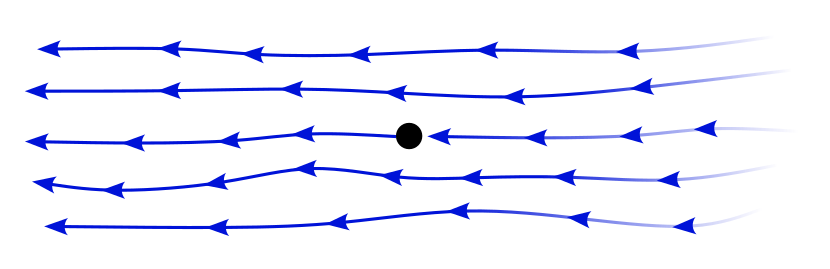}
\end{overpic}
\caption[An (almost) tubular flow at $\infty$.]{\textit{A $0$--index fixed point.}}\label{ind}

\end{figure}

    When $d=\sum_{i=1}^nd_i=\pm1$ the situation is similar, i.e., using a similar argument to the one above, we make $\infty$ into a fixed point for the flow of Poincaré index $-d$ - where $d$ is either $1$ or $-1$. We now recall that given a Saddle Focus fixed point $p$ for a flow, its Poincaré index will be either $1$ or $-1$ - depending on the sign of the Jacobian determinant at $p$. Therefore, using a similar argument to the one above, Cor.\ref{zeroone} now follows.
\end{proof}

Having proven Cor.\ref{zeroone}, we are finally ready to state and prove Theorem \ref{in} from the Introduction. From now on, given any smooth vector field $F$ of $\mathbb{R}^3$, denote by $\phi_s(t)$ the solution curve passing through $s$ at times $t$ - parameterized such that $\phi_s(0)=s$. More generally, given any set $S\subseteq\mathbb{R}^3$, we will often denote by $\phi_t(S)$ its image by the flow at time $t$ - and given an interval $(a,b)$, we will often denote the collection of flow lines connecting $\cup_{s\in S}\phi_{a}(s)$ and $\cup_{s\in S}\phi_b(s)$ by $\phi_{(a,b)}(S)$. In addition, recall that given any compact curve $\gamma\subseteq\mathbb{R}^3$ such that for all $s\in\gamma$, $F(s)\ne0$, there exists some $t>0$ for which $\phi_{(-t,t)}(\gamma)$ forms a surface. With those ideas in mind, we now prove:
\begin{theorem}
  \label{infinitytheorem}  Let $\dot{s}=F(s)$, $s=(x,y,z),F=(F_1,F_2,F_3)$ be a smooth vector field of $\mathbb{R}^3$ satisfying the following geometric scenario (see Fig.\ref{comp}):
    \begin{enumerate}
        \item $F$ extends continuously to $\infty$ - where $\infty$ is added as a fixed point for the flow whose Poincaré index is either $0$ or $\pm 1$.
        \item $F$ has a finite number of fixed points in $\mathbb{R}^3$, all of which are non-degenerate and have complex-conjugate eigenvalues.
        \item There exists some $i\in\{1,2,3\}$ such that the velocity level set $H=\{s\in\mathbb{R}^3\mid F_i(s)=0\}$ satisfies the following (see Fig.\ref{comp}):
        \begin{itemize}

            \item $H$ is unbounded, and homeomorphic to a plane.
            \item Let $l$ denote the tangency set of the vector field $F$ to $H$. Then, $l$ is homeomorphic to $\mathbb{R}$.
            \item Parameterizing $l=(s_1(t),s_2(t),s_3(t))$, $t\in\mathbb{R}$, then $s_i(t)$ is an increasing function in $t$.
            \item For all $c\in\mathbb{R}$, $H$ is transverse to the plane $H_c=\{s\in\mathbb{R}^3\mid s_i=c\}$ - and in particular, $l_c=H_c\cap H$ is unbounded and homeomorphic to $\mathbb{R}$. Moreover, $l_c\cap l$ is a singleton.
            \item Let $I$ be a segment on $l$ such that there are no fixed points on $I$. Then, for all $s\in I$ there exists some $\epsilon>0$ depending on $s$ such that the surface $\cup_{s\in I}\phi_{(-\epsilon,\epsilon)}(s)$ is either in $\{F_i(s)\geq0\}$ or $\{F_i(s)\leq0\}$.
            \item Finally, $H\setminus l$ is composed of two (topological) half-planes, $H_+$ and $H_-$ such that on $H_+$ initial conditions cross from $\{F_i(s)\leq0\}$ into $\{F_i(s)>0\}$, and on $H_-$ from $\{F_i(s)\geq0\}$ to $\{F_i(s)<0\}$.

        \end{itemize}
    \end{enumerate}

Then, there exist at least two fixed points $x_1,x_2$ (which possibly coincide) such that each fixed point generates a one-dimensional invariant manifold, $\Gamma_i$, $i=1,2$ (respectively) connecting it to $\infty$ (see Fig.\ref{conn}). Moreover, there exists a curve $\gamma\subseteq\mathbb{R}^3$ such that the union $\{x_1,x_2,\infty\}\cup\Gamma_1\cup\Gamma_2\cup\gamma$ is the unknot.  
\end{theorem}
\begin{proof}
We first prove the existence of the fixed point $x_1$ and the invariant manifold $\Gamma_1$ - the proof for the existence of $x_2$ and $\Gamma_2$ is similar, and we will indicate how it is done towards the end of the proof. Before giving a sketch of what lies ahead, we first note that given $F,H$ and $l$ as above, there are precisely two possibilities:
\begin{itemize}
    \item \textbf{The non-generic case} - $F$ is tangent to some unbounded arc $\gamma\subseteq l$ connecting some fixed point $p$ and $\infty$.
    \item \textbf{The generic case} - there is no such arc $\gamma\subseteq l$ as described above such that $F$ is tangent to $\gamma$.
\end{itemize}

It is easy to see that whenever there exists some $\gamma$ as in the first possibility, $\gamma$ has to be an unbounded flow line connecting $\infty$ to $p$ - which makes $\gamma$ some invariant manifold connecting $p$ and $\infty$ (and being a flow line, by definition it is one-dimensional). Or, in other words, whenever the non-generic case holds the conclusion of Theorem \ref{infinitytheorem} follows by setting $\Gamma_1=\gamma$, $x_1=p$. Similarly, by considering $l\setminus\gamma$ a similar argument accounts for the existence of $x_2$ and $\Gamma_2$ in the non-generic case. Therefore, to prove Theorem \ref{infinitytheorem} it would suffice to prove it in the generic case - i.e., from now on we always implicitly assume there is no unbounded arc $\gamma\subseteq l$ connecting any fixed-point $x$ to $\infty$ such that $F$ is tangent to $\gamma$. \\

Having said that, we now give a (very) brief overview of the proof of Theorem \ref{infinitytheorem} for the generic case. The proof itself will be rather technical, and as such we divide it into several stages:

\begin{itemize}
    \item In Stage $I$ we begin by performing basic qualitative analysis of the vector field $F$, based on its properties in the premise of Theorem \ref{infinitytheorem}.
    \item In Stage $II$ we use the said analysis to prove the existence of $\Gamma_1$ under idealized assumptions. We do so by constructing three-dimensional bodies using flow lines emanating from $l$, from which trajectories can only escape.
    \item In Stage $III$ we prove the idealized assumptions can be removed, thus implying the general existence of the invariant manifold $\Gamma_1$.
    \item Finally, in Stage $VI$ we indicate how the same ideas and analogous arguments imply the existence of $\Gamma_2$, as well as analyze the knotting properties of $\Gamma_1$ and $\Gamma_2$. This concludes the proof.
\end{itemize}

\begin{figure}[h]
\centering
\begin{overpic}[width=0.5\textwidth]{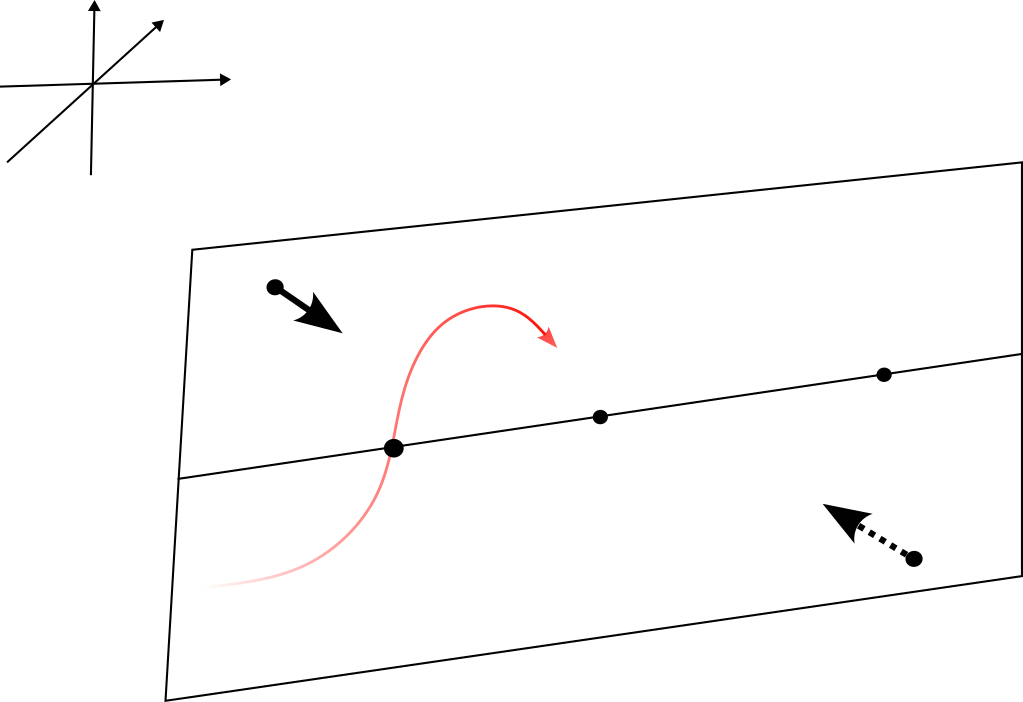}
\put(550,230){$x_1$}
\put(820,340){$x_2$}
\put(450,290){$l_1$}
\put(920,350){$l_2$}
\put(710,320){$l_3$}
\put(720,450){$H_+$}
\put(210,70){$H_-$}
\put(230,595){$x$}
\put(165,680){$y$}
\put(75,700){$z$}
\end{overpic}
\caption{\textit{A sketch of $H=H_+\cup H_-,l_1,l_2,x_1$ and $x_2$ in Case $A$ - along with a red flow line tangent to $l_2$ - along with the direction of the vector field $F$ on $H$. By definition, $l$ is the straight line corresponding to $l_1\cup l_2\cup l_3$.}}\label{comp}

\end{figure}

\subsection{Stage I - basic qualitative analysis}. 

As indicated above, in this section we perform basic qualitative analysis of the vector field $F$. To do so, recall the level set $H=\{s\in\mathbb{R}^3\mid F_i(s)=0\}$ defined above. By assumption, this set is unbounded and homeomorphic to a plane (as illustrated in Fig.\ref{comp}). Without any loss of generality, from now on we always assume $i=1$, i.e., that we have the equality $H=\{s\mid \dot{x}(s)=0\}$, defined by $s=(x,y,z)$, and $(F_1(s),F_2(s),F_3(s))=(\dot{x}(s),\dot{y}(s),\dot{z}(s))$. To continue, recall that per our assumption on $F$ we know the following holds:
\begin{itemize}
    \item Given any $c\in\mathbb{R}$, $H$ is transverse to the plane $H_c=\{(c,y,z)\mid y,z\in\mathbb{R}\}$.
    \item For all $c\in\mathbb{R}$, $H\setminus H_c$ is divided into two components (see Fig.\ref{comp2}).
    \item Recall $l$, the tangency line of $F$ to $H$ - then, $H_c\cap l$ is a singleton, corresponding to the transverse intersection of the curves $H\cap H_c$ and $l$.
    \item Parameterizing $l=(s_1(t),s_2(t),s_3(t))$, $t\in\mathbb{R}$, then $s_1$ is increasing with $t$.
\end{itemize}

As $l$ is the tangency set of $F$ to the set $\{s\mid \dot{x}(s)=0\}=H$ and since $F$ only has a finite number of fixed points, it follows all the fixed points of $F$ lie on $l$. Similarly, as $s_1$ is an increasing function there must be a fixed point $x_1\in l$ which minimizes the $x$--coordinate among all the fixed points. As such, from now on we choose $x_1$ to be that fixed point (later on, we analogously choose $x_2$ to be the point which maximizes the $x$--coordinate). To continue, write $x_1=(c_1,c_2,c_3)$, and let $l_1\subseteq l$ denote the sub-curve of $l\cap\{(x,y,z)\mid x<c_1\}$ connecting $x_1$ to $\infty$ (it exists per our assumptions on $l$). Per the discussion above, we know $l_1$ is not a flow line. Moreover, since $x_1$ minimizes the $x$ coordinates of all the fixed points, it follows that for all $s\in l_1$ we have $F(s)\ne0$ (see Fig.\ref{comp}).\\

To continue, set $H_1=\{(c_1,y,z)\mid y,z\in\mathbb{R}\}$. By the above $H_1$ is transverse to $H$ and $H\cap H_1$ has to be a curve homeomorphic to $\mathbb{R}$ (see Fig.\ref{comp2}). Now, recall we denote the inner product by $\bullet$. As the normal vector to $H_1$ is $(1,0,0)$ and recalling  $H=\{s\mid \dot{x}(s)=0\}=H$ (i.e., $H$ is the vanishing set of $\dot{x}(s)$ which is defined by the function $F_1(s)$), we conclude two things (see Fig.\ref{comp2}):

\begin{itemize}
    \item  For $s\in H_1\cap\{F_1(s)>0\}$ we have $F(s)\bullet(1,0,0)=F_1(s)>0$ - in other words, on the half-plane $H_1\cap\{F_1(s)>0\}$ the vector field $F$ points in the $(1,0,0)$ direction (but not necessarily parallel to it).
    \item Similarly, for $s\in H_1\cap\{F_1(s)<0\}$ we have $F(s)\bullet(1,0,0)<0$ - in other words, on the half-plane $H_1\cap\{F_1(s)<0\}$ the vector field $F$ points in the $(-1,0,0)$ direction (and again, not necessarily parallel to it).
\end{itemize}

\begin{figure}[h]
\centering
\begin{overpic}[width=0.5\textwidth]{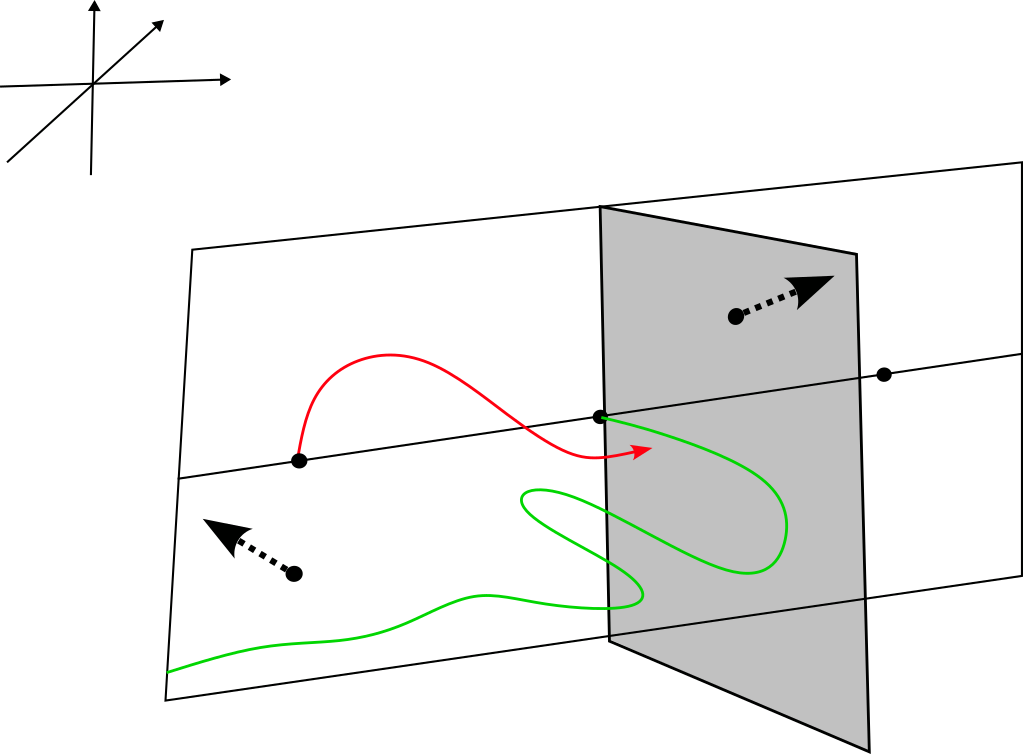}
\put(545,350){$x_1$}
\put(820,340){$x_2$}
\put(400,330){$l_1$}
\put(920,350){$l_2$}
\put(670,450){$H_1$}
\put(180,150){$H_-$}
\put(230,645){$x$}
\put(165,720){$y$}
\put(80,745){$z$}
\end{overpic}
\caption{\textit{Case $A$ - the plane $H_1$ is transverse to $H$, and the forward trajectory of every initial condition in $l_1$ eventually hits either $H_-$ or $H'=H_1\cap\{F_1(s)>0\}$ transversely (where $H'$ is the shadowed half-plane). The green curve denotes the points where initial conditions from $l_1$ hit $H_-\cup H_1$. On $H_1$ the vector field points into $\{x>c_1\}$.}}\label{comp2}

\end{figure}

To continue, recall that given any $s\in\mathbb{R}^3$ we denote its flow line with respect to $F$ by $\phi_t(s)$, parameterized such that $\phi_0(s)=s$. Further recall that by our assumptions on the vector field $F$, for all $s\in l_1$ there exists some maximal time interval $\infty\geq t(s)>0$ such that the union of flow lines $\cup_{s\in l_1}\phi_{(-t(s),t(s))}(s)$ lies wholly inside either $\{F_1(s)\geq0\}$ or $\{F_1(s)\leq0\}$ (see Fig.\ref{comp2}). Additionally, recall we denote by $H_+$ the component of $H\setminus l$ at which trajectories cross from $\{\dot{x}\leq0\}$ into $\{\dot{x}>0\}$ - and by $H_-$ the component of $H\setminus l$ at which trajectories cross from $\{\dot{x}\geq0\}$ to $\{\dot{x}<0\}$ (see Fig.\ref{comp}). This implies precisely one of the following must hold throughout $l_1$ (see Fig.\ref{comp2}):

\begin{enumerate}
    \item \textbf{\label{time} Case $A$ -} for all $s\in l_1$, we have $\phi_{(-t(s),t(s))}(s)\subseteq\{\dot{x}(s)\geq0\}$. In this case, the forward trajectory of every $s\in l_1$ either hits transversely the (topological) half-plane $H_-$ and enters $\{\dot{x}(s)<0\}$, or it hits transversely $H_1\cap \{\dot{x}(s)>0\}$ and enters the region $\{x>c_1\}\cap \{F_1(s)>0\}$ (see Fig.\ref{comp2}).
    \item \textbf{Case $B$} - for all $s\in l_1$ the curve $\phi_{(-t(s),t(s))}(s)$ is in $\{\dot{x}(s)\leq0\}$. In this case the backwards trajectory of every $s\in l_1$ either hits transversely $H_-$ and enters $\{\dot{x}(s)>0\}$ in backwards time, or it hits transversely $H_1\cap \{\dot{x}(s)<0\}$ and enters $\{x>c_1\}\cap \{F_1(s)<0\}$ in backwards time.
\end{enumerate}

In what follows we will only show how to deal with Case $A$ - the proof for Case $B$ is very similar, and will be briefly sketched in Stage $III$ of the proof. To motivate our argument, for every $s\in l_1$ define $p(s)$ to be the first positive time such that $\phi_{p(s)}(s)\in H_-\cup(H_1\cap\{F_1(s)>0\})$.  Also, from now on set $V=\cup_{s\in l_1}\phi_{p(s)}(s)$ and set $Q_1=H_-\cup(H_1\cap\{F_1(s)>0\})$. By the above, $F$ is transverse to $Q_1$, and denoting by $R_1 \{\dot{x}\geq0\}\cap\{x<c_1\}$ the quadrant trapped between $H$ and $H_1$, it is clear that $Q_1\subseteq \partial R_1$, and that $Q_1$ is the maximal set on $\partial R_1$ on which $F$ points outside of $R_1$. In addition, by the definition of Case $A$, for all $s\in l_ 1$ the time function $p(s)$ is well-defined and non-zero (see Fig.\ref{comp2}). This implies the existence of a first-hit map $f:l_1\to V\cap Q_1$ defined by $f(s)=\phi_{p(s)}(s)$. Since $F$ is transverse to $Q_1$ and as $\phi_{p(s)}(s)\in Q_1$ for all $s\in l_1$, it follows that when the trajectory of $s$ flows to $\phi_{p(s)}(s)$ without any tangent intersections with $l_1$, $f$ is continuous around $s$. As we will see in the next section, the analysis of this function $f$ (and in particular, its discontinuities) will form the backbone for the proof of Theorem \ref{infinitytheorem}.\\

\begin{figure}[h]
\centering
\begin{overpic}[width=0.35\textwidth]{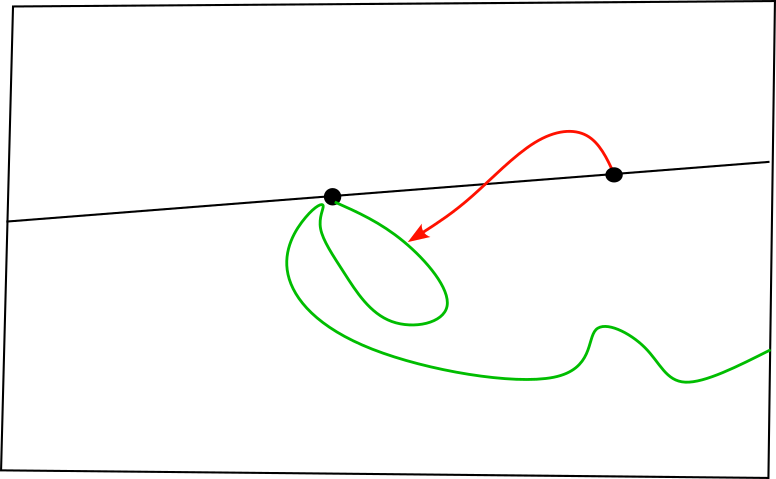}
\put(245,380){$l_2$}
\put(400,390){$\infty$}
\put(920,425){$l_1$}
\put(180,150){$H_-$}

\end{overpic}
\caption{\textit{The idealized assumptions - the set $V$ connects with $Q_1$ in a green curve in some neighborhood of $\infty$.}}\label{cone}

\end{figure}
\subsection{Stage II - proving the Theorem under idealized conditions.}
We are now ready to prove the existence of the invariant manifold $\Gamma_1$ for the fixed point $x_1$ in Case $A$, which we first do under idealized assumptions on $F$ (we will deal with the general case in Stage $III$ of the proof). To motivate this idealized scenario, note that since $l_1$ connects $x_1$ and $\infty$, in the ideal case we would expect $V$ to be "well-behaved" around $\infty$. More precisely, in the ideal sccenario we would expect $f$ and $x_1$ to satisfy the following two conditions (see Fig.\ref{cone}):

\begin{itemize}
    \item  $f$ is continuous around $x_1$ and satisfies $\lim_{s\to x_1}f(s)=x_1$.
    \item $\lim_{s\to\infty} f(s)=\infty$, i.e., we can extend $f$ to $\infty$ by setting $f(\infty)=\infty$.
    \item $f$ is continuous on some neighborhood of $\infty$ in $l_1$.
    \item The Jacobian matrix at $x_1$ has no purely imaginary eigenvalues - i.e., the two complex eigenvalues have non-zero real part, which makes $x_1$ either a saddle focus or a complex sink or source.
\end{itemize}

\begin{figure}[h]
\centering
\begin{overpic}[width=0.35\textwidth]{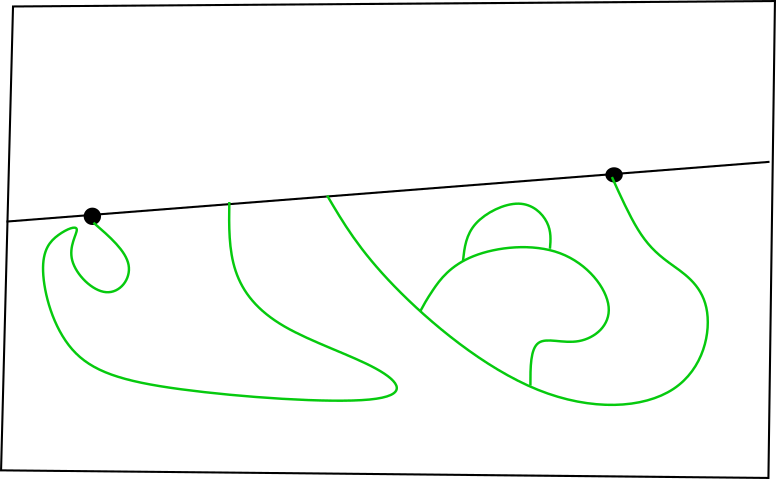}
\put(100,370){$\infty$}
\put(400,390){$l_1$}
\put(800,425){$x_1$}
\put(180,150){$H_-$}

\end{overpic}
\caption{\textit{The green curve denotes the intersection of $\overline{V}$ with $Q_1$ - in this scenario $V$ is not a surface and the curve is branched (that being said, $V$ is still a two-dimensional set). For simplicity, in this scenario we assume this intersection lies entirely in $H_-$.}}
\label{disc2}

\end{figure}

From now on we always refer to these four assumptions as the \textbf{idealized assumptions}, and as remarked earlier, in this stage of the proof we prove the existence of $\Gamma_1$ under the assumption that $F$ also satisfies these extra assumptions. To begin, we first note the set $V$ is a two-dimensional set composed of all the flow lines connecting $l_1$ and $f(l_1)\subseteq Q_1$. Per the discussion above, we conclude that under the idealized assumptions there are exactly two possible scenarios:

\begin{enumerate}
    \item  $f(l_1)$ is a curve in $Q_1$ beginning at $x_1$ and ending at $\infty$, which does not intersect $l_1$ - in other words, $f(l_1)$ is homeomorphic to an open interval with one end at $x_1$ and another at $\infty$. In this case, it is clear $V$ is a surface homeomorphic to a disc, composed of all the flow lines connecting $l_1$ to $f(l_1)$ (see Fig.\ref{disc1}). As $F$ is transverse to $Q_1$, it would follow that in this scenario, $f:l_1\to Q_1$ is continuous throught $l_1$.
    \item $f(l_1)$ is not a connected curve, but rather a disjoint collection of curves. By the discussion at the end of Stage $I$, this can only occur if  the trajectory of some initial condition $s\in l_1$ hits $l_1$ tangently before hitting $Q_1$ transversely at $f(s)$ (see Fig.\ref{disc2}). That being said, in this case the set $V$ would still be two dimensional (even if not a surface).
\end{enumerate}

\begin{figure}[h]
\centering
\begin{overpic}[width=0.35\textwidth]{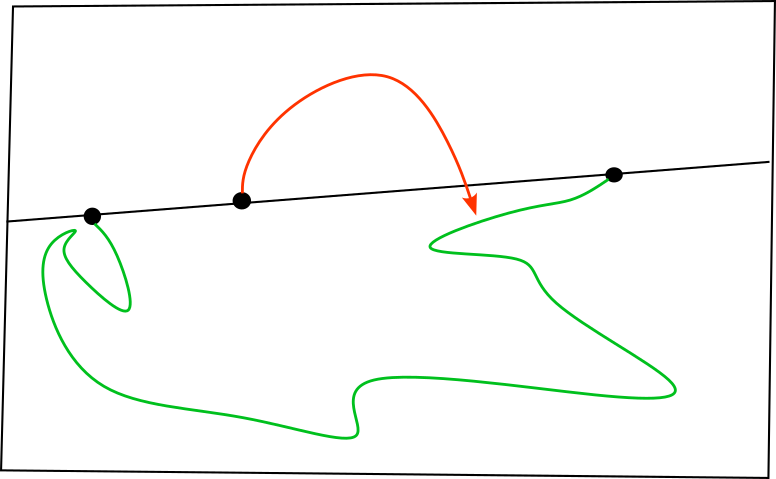}
\put(100,370){$\infty$}
\put(400,390){$l_1$}
\put(800,425){$x_1$}
\put(180,150){$H_-$}

\end{overpic}
\caption{\textit{The green curve denotes the intersection of $\overline{V}$ with $Q_1$ - in this scenario $V$ is a surface hence the curve is homeomorphic to an interval. For simplicity, in this scenario we assume this intersection lies entirely in $H_-$.}}\label{disc1}

\end{figure}

We claim that in each of these two scenarios, the set $V$ and $Q_1$ trap between themselves a three-dimensional body $C_1$ whose boundary is made either of flow lines in $V$ or points on $Q_1$ (later, we will prove no trajectories can enter $C_1$). In the first scenario, i.e., when $f(l_1)$ is a connected curve, this is immediate, as $C_1$ is just the three-dimensional set trapped between $Q_1$ and the flow lines on $V$ connecting initial conditions in $l_1$ to $f(l_1)$ (see Fig.\ref{cone}). On the other hand, when $f$ has discontinuities as in the second scenario, we need a more technical argument, which we will now prove.\\ 

\begin{figure}[h]
\centering
\begin{overpic}[width=0.5\textwidth]{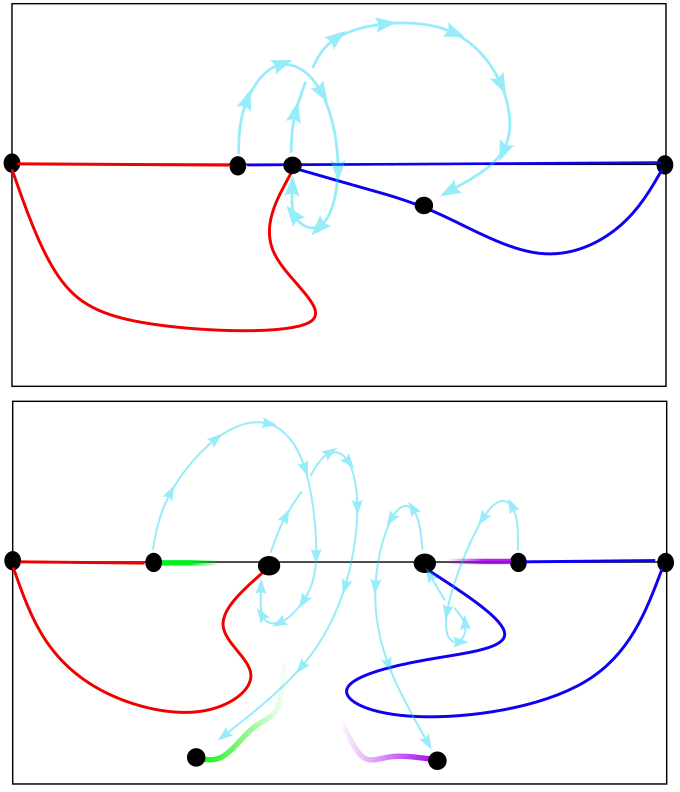}
\put(870,850){$H_+$}
\put(870,650){$H_-$}
\put(870,350){$H_+$}
\put(870,150){$H_-$}
\put(40,265){$\infty$}
\put(40,765){$\infty$}
\put(790,265){$x_1$}
\put(790,765){$x_1$}

\put(500,680){$f(v)$}

\put(180,310){$v_\infty$}
\put(290,810){$v$}
\put(360,810){$s$}
\put(655,310){$v_1$}
\put(320,310){$s_\infty$}
\put(530,310){$s_1$}
\put(570,20){$f(v_1)$}
\put(135,20){$f(v_\infty)$}

\put(190,810){$I_\infty$}
\put(500,810){$I_1$}
\put(80,310){$I_\infty$}
\put(720,310){$I_1$}

\put(230,615){$M_\infty$}
\put(680,715){$M_1$}
\put(190,135){$M_\infty$}
\put(680,140){$M_1$}

\end{overpic}
\caption{\textit{In both images, for simplicity, $M_\infty=f(I_\infty)=D_\infty$ and $M(I_1)=f(I_1)=D_1$, where $I_\infty,I_1$ are denoted by the red and blue arcs (respectively) on the line $l_1$ which separates $H_-,H_+$, while the cyan directed curves denote flow lines. In the upper image, the point $v$ separates the intervals $I_1,I_\infty$, and the tangency of its trajectory to $l_1$ at $s$ causes a removable singularity (note that by definition, $f(s)=f(v)$). In the lower image on the other hand, the tangency of the trajectory of both $v_\infty$ and $v_1$ causes a jump discontinuity, i.e., $f(l_1)$ has several components which lie away from one another (in this scenario, $f(v_\infty)=f(s_\infty)$ and $f(v_1)=f(s_1)$). In the lower image, the green and purple arcs denote segments on $l_1$ adjacent to $I_\infty$ and $I_1$. For simplicity, both scenarios are sketched on the plane $H$ instead of $Q_1$.}}
\label{disc0}

\end{figure} 

Therefore from now on assume $f$ has discontinuities in $l_1$. By the continuity of $f$ around $x_1$ and $\infty$ it follows we can partition $l_1$ into a finite number of segments on which $f$ is continuous, save at their boundary points. Now, let $I_1,I_2$ denote open segments on $l_1$ s.t. there exists some $s\in l_1$ for which $I_1\cup{s}\cup I_2$ is also an open segments on $l_1$. Further assume $s$ is such that $f$ is continuous on both $I_1$ and $I_2$, and not in $s$, i.e., $s$ is a discontinuity point for $f$ in $l_1$. As discussed earlier, since $f$ is a local first-hit map discontinuity points like $s$ correspond to initial conditions whose trajectory w.r.t. the flow hits $l_1$ tangently at some $\overline{s}$ before hitting $Q_1$ transversely at $f(s)$. This implies there are precisely two possibilities regarding the nature of the discontinuity at $s$:
\begin{itemize}
    \item $f$ has a removable discontinuity at $s$. More precisely, the flow lines connecting $I_1\cup I_2\cup\{s\}$ to $f(I_1)\cup f(I_2)\cup\{\overline{s}\}$ satisfy $\lim_{r_1\to s}f(r_1)=\overline{s}=\lim_{r_2\to s}f(r_2)$, $r_2\in I_2, r_1\in I_1$ (see the illustration in Fig.\ref{disc0}). By definition, $f(s)\ne \overline{s}$.
    \item $f$ has a jump discontinuity, i.e., $\lim_{r_1\to s}f(r_1)\ne\lim_{r_2\to s}f(r_2)$ where $r_2\in I_2$, $r_1\in I_1$ (see the illustration in Fig.\ref{disc0}). This corresponds to the case where, say, all initial conditions in $I_1$ flow together with the trajectory of $s$ and hit $Q_1$ transversely (and continuously) away from $f(I_2)$. 
\end{itemize}

Now, let $M_1,M_\infty$ denote the components of ${f(l_1)}$ which connect with $x_1$ and $\infty$, respectively (see Fig.\ref{disc0}). Since $f$ is discontinuous at least in one point at the interior of $l_1$, two such components must exist. Let $I_1,I_\infty$ denote the segments of $l_1$ with $x_1,\infty$ on their boundary (respectively) s.t. $f(I_i)\subseteq M_i$, $i=1,\infty$ - it is immediate that $f(I_i)$ are mapped to arcs on $M_i$, denoted by $D_i$, where again $i=1,\infty$ (see the illustration in Fig.\ref{disc0}). By the maximality of components and the continuity of $f$ at transverse intersections, it follows $D_1$ connects $x_1$ to some point $s_1\in l_1$ through $Q_1$. Likewise, $D_\infty$ connects $\infty$ with some point $s_\infty$ through $Q_1$. Now, let $v_1,v_\infty$ denote the respective boundary points for $I_1,I_\infty$ which flow to the points $s_1,s_\infty\in l_1$ and hit them tangently before continuing to $f(v_1),f(v_\infty)$ (note that by definition, $f(v_i)=f(s_i),i=1,\infty$ - see the illustration in Fig.\ref{disc0}).\\

There are now two possibilities to consider - either $v_1=v_\infty=v$ for some $v\in l_1$, or $v_1\ne v_\infty$. We first consider the case where $v_1=v_\infty=v$, for some $v\in l_1$. By definition, it follows $s_1=s_\infty=s$ for some $s\in l_1$. Consequently, we have both $l_1=I_1\cup\{v\}\cup I_\infty$ and $D_i=M_i$, $i=1,\infty$ (see the illustration in the upper image in Fig.\ref{disc0}). We now claim $v$ must be a removable discontinuity. To see why, make a slight smooth modification in some arbitrarily small neighborhood $N$ of $s$ s.t. the flow lines emanating from some neighborhood of $v$ on $l_1$ hit $Q_1$ transversely (we can do so as $v$ is the only discontinuity point on $l_1$). In particular, we do not change the trajectories of initial conditions on $l_1$ that do not enter $N$, and we do not add any tangency points with $l_1$.\\

Since $v$ was the unique discontinuity point, removing it as we just did makes $f$ continuous throughout $l_1$. Moreover, as this deformation can be chosen to arbitrarily $C^k$--small from the original vector field $F$, $k\geq1$, it follows that the curve $f(l_1)$ is the uniform limit of curves defined by continuous homeomorphisms $f':l_1\to Q_1$. As such, it follows the only possibility is that $v$ forms a removable discontinuity, i.e., $\lim_{r_1\to v}f(r_1)=\lim_{r_\infty\to v}f(r_\infty)=s$, where $s\ne f(v)$ but $f(s)=f(v)$, as illustrated in the upper image of Fig.\ref{disc0}. Similarly to the case where $f$ is continuous on $l_1$ we conclude there exists a three dimensional body $C_1$ trapped between $V$ and $Q_1$, where $V$ denotes the flow lines connecting $l_1$ to $f(l_1)$ - again, $\partial C_1$ would be composed of regions on $Q_1$ and flow lines on $V$. As such, the proof of existence of $C_1$ in the case when $v_1=v_\infty$ is complete.\\

\begin{figure}[h]
\centering
\begin{overpic}[width=0.6\textwidth]{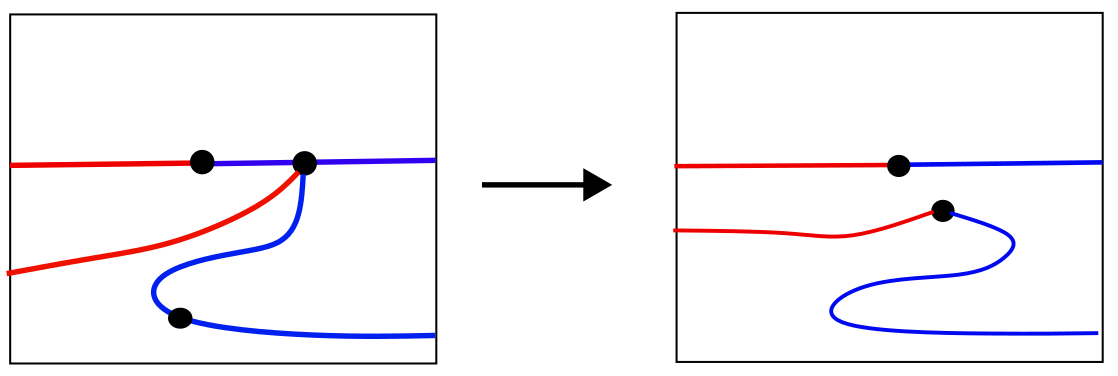}

\put(960,210){$B$}
\put(800,100){$f(v)$}
\put(40,210){$R$}

\put(170,210){$v$}
\put(790,210){$v$}
\put(260,210){$v'$}
\put(620,210){$J$}
\put(350,210){$B$}

\put(90,20){$f(v)$}
\put(800,0){$f(B)$}

\put(80,130){$f(R)$}
\put(270,135){$f(B)$}
\put(630,150){$f(R)$}

\end{overpic}
\caption{\textit{Perturbing a removable discontinuity for $f$ on the left into a continuity point on the right. In the illustration of the left, the trajectory of $v$ hits $l_1$ on tangently at $v'$ after which it flows to $f(v)$. By definition, $f(v)=f(v')$}}
\label{disc}

\end{figure} 

We now prove the existence of $C_1$ under the assumption $v_1\ne v_\infty$, i.e., we now assume there exist multiple discontinuities on $l_1$. As there are no other possibilities, this will prove that whatever the case is, under the idealized assumptions the body $C_1$ always exists, and its boundary is composed of flow lines on $V$ and regions in $Q_1$. To this end, we first claim it is enough to consider the case where all discontinuities for $f$ on $l_1$ are jump discontinuities. To justify this, note that if $v\in l_1$ is a removable discontinuity, then by an arbitrarily small smooth perturbation of $F$ we can eliminate it (see the illustration in Fig.\ref{disc}). In more detail, if the trajectory of $v$ hits $l_1$ tangently at some $v'$ and creates a removable discontinuity, we smoothly push the trajectory connecting $v,v'$ downwards s.t. the intersection becomes transverse and with $Q_1$.\\

As we can modify $F$ in this way around all removable discontinuities, it follows we can destroy all removable discontinuities on $l_1$ and make them continuous for $f$. Moreover, it is clear the modified vector field $F'$ can be constructed s.t. its $C^k$ distance from $F$ is arbitrarily small, where $k\geq1$ - i.e., can be we can $C^k$ approximate $F$ by vector fields for which the only possible discontinuities for $f$ are jump discontinuities (where $k\geq1$). It follows that if the three-dimensional body $C_1$ exists for all of these $C^k$ approximations of $F$, it must also exist for $F$ as well. In other words, to solve the case where $v_1\ne v_\infty$ it is enough to solve it only for the case when $f$ only has jump discontinuities on $l_1$ (in particular, $v_1$ and $v_\infty$ would also be jump discontinuities).\\

\begin{figure}[h]
\centering
\begin{overpic}[width=0.6\textwidth]{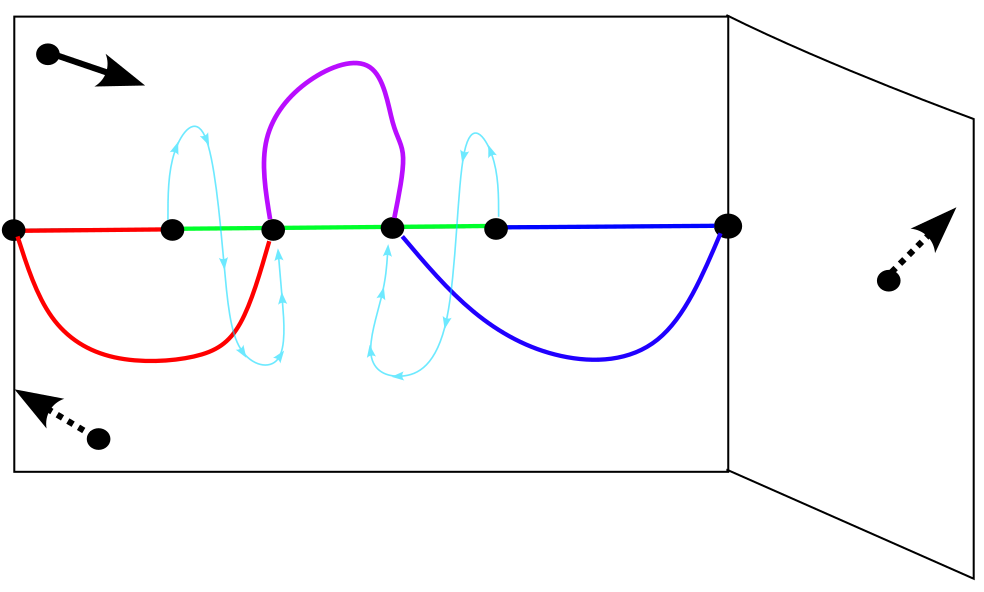}

\put(1000,210){$H_1$}
\put(-40,480){$H_+$}
\put(-40,210){$H_-$}

\put(140,180){$f(I_\infty)$}
\put(180,385){$v_\infty$}
\put(280,390){$s_\infty$}
\put(620,190){$f(I_1)$}
\put(410,390){$s_1$}
\put(480,390){$v_1$}

\put(330,550){$\mathcal{I}$}
\put(330,320){$A$}
\put(630,385){$I_1$}
\put(90,380){$I_\infty$}

\end{overpic}
\caption{\textit{The arc $A$ stretched on $l_1$ between $s_1,s_\infty$. The green arc denotes $I$ (in this illustration, $A\subseteq I$). The directions of the vector field on $H_1,H_-$ and $H_+$ are also sketched. Again, $\mathcal{I}=\cup_{s\in I}\phi_{t(s)}(s))$, and $\mathcal{I}\cup A$ is the boundary of a topological disc }}
\label{disc01}

\end{figure} 

Therefore, from now on we will assume there exist multiple discontinuities for $f$ on $l_1$, all of which are jump discontinuities. Let $I$ denote the segment of $l_1$ lying between $I_1,I_\infty$, i.e., $I=l_1\setminus(I_1\cup I_\infty)$ (see Fig.\ref{disc01}). By definition, the boundary points of $I$ in $l_1$ are $v_1$ and $v_\infty$ (see the illustration in Fig.\ref{disc01}). We consider the trajectories of both initial conditions $v_1$ and $v_\infty$. After leaving $v_1$ and $v_\infty$, these trajectories spiral inside the open quadrant $R_1$ defined by $R_1=\{(x,y,z)\mid x<c_1\}\cap\{(x,y,z)\mid \dot{x}>0\}$ until hitting $s_1,s_\infty$ tangently, after which they remain trapped within $\overline{R_1}$ until hitting the points $f(s_1),f(s_\infty)\in Q_1$ transversely (as illustrated in Fig.\ref{disc01}). Now, recall that we denote the flow function at time $t$ by $\phi_t$, and let $S$ denote a topological disc in $R_1$ whose boundary is given by the union $\partial S=(\cup_{s\in I}\phi_{t(s)}(s))\cup A$, where the components of $\partial S$ are defined as follows:
\begin{itemize}
    \item Set $\mathbb{R}^+=\{r\in\mathbb{R}|r\geq0\}$. $t(s):\overline{I}\to\mathbb{R}^+$ Denotes a continuous map s.t. when $s$ is interior to $I$, $t(s)>0$ and $\phi_{t(s)}(s)$ is interior to the flow line between $s$ and its first intersection point with $Q_1\cup l_1$. As far as its boundary behavior is concerned, we choose $t$ s.t. $\phi_{t(v_\infty)}(v_\infty)=s_\infty$ and likewise, $\phi_{t(v_1)}(v_1)=s_1$.
    \item $A$ is the sub-arc on $l_1$ with endpoints $s_1,s_\infty$ - see the illustration in Fig.\ref{disc01}.
\end{itemize}

Let us now consider the flow lines connecting $\partial S$ to $Q_1$. It is immediate that the flow lines connecting $\partial S$ to $Q_1$ must hit $Q_1$ in a way that defines a collection of Jordan domains, $T_1,...T_k\subseteq Q_1$ - where the boundaries of $T_1,...,T_k$ on $Q_1\cup l_1$ are composed of components of $f(I)$ and arcs on $l_1$, as illustrated in Fig.\ref{disc02}. \\

\begin{figure}[h]
\centering
\begin{overpic}[width=0.6\textwidth]{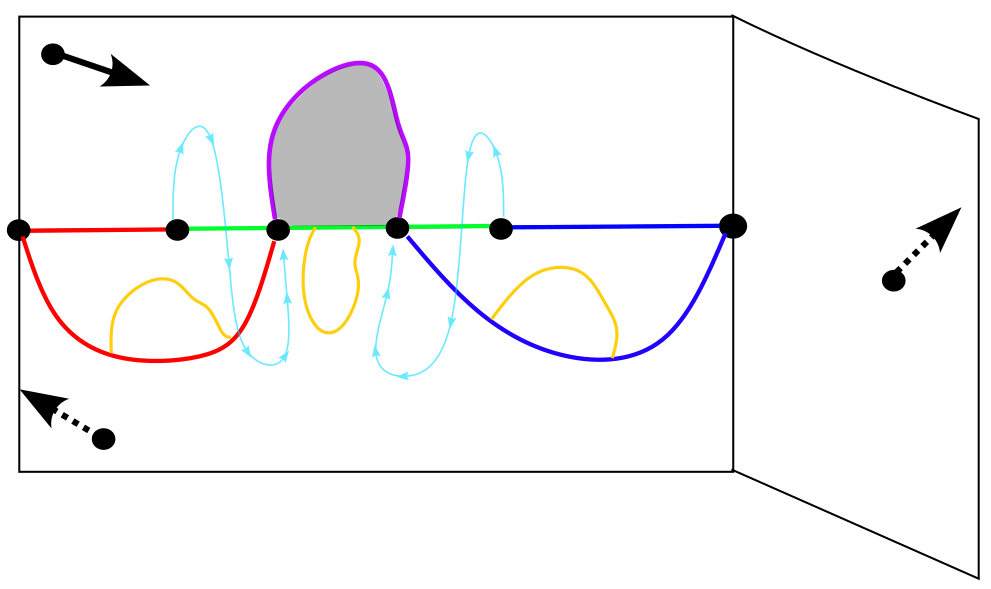}

\put(1000,210){$H_1$}
\put(-40,480){$H_+$}
\put(-40,210){$H_-$}
\put(550,270){$T_3$}
\put(140,270){$T_2$}
\put(140,190){$f(I_\infty)$}
\put(180,385){$v_\infty$}
\put(280,390){$s_\infty$}
\put(350,380){$A$}
\put(620,190){$f(I_1)$}
\put(410,390){$s_1$}
\put(480,390){$v_1$}

\put(330,550){$\mathcal{I}$}
\put(320,310){$T_1$}
\put(340,430){$S$}
\put(630,385){$I_1$}
\put(90,380){$I_\infty$}

\end{overpic}
\caption{\textit{The domains $T_1,...,T_k$ - in this scenario, $r=1$, $k=3$. The shaded area is the surface $S$ - despite its flat appearance, it should be thought of as a surface inside a three-dimensional space (in particular, $S$ is \textbf{not} a region on $H_+$). The yellow arcs on $\partial T_i$, $i=1,2,3$ correspond to the images of sub-arcs on $I$ under $f$. $\mathcal{I}$ denotes $\cup_{s\in I}\phi_{t(s)}(s))$}}
\label{disc02}

\end{figure} 

We now distinguish between two possible types of such Jordan domains. Of these Jordan domains, let $T_1,...,T_r$, $r\leq k$ denote the Jordan domains among $T_1,...,T_k$ satisfying the following (whenever they exist):
\begin{itemize}
    \item For every $1\leq j\leq r$, $\partial T_j\setminus Q_1$ is composed of arcs on the segment $A$, as illustrated in Fig.\ref{disc02}.
    \item For all $1\leq j\leq r$, every point $s_0\in \partial T_j\cap Q_1$ is connected by a flow line $\gamma$ to some $\omega_0\in \cup_{s\in I}\phi_{t(s)}(s)$ (see Fig.\ref{disc02}). Moreover, $s_0=f(r_0)$, where $\omega_0=\phi_{t(r_0)}(r_0)$.
\end{itemize}

Put simply, $T_1,...,T_r$ are the Jordan domains created when initial conditions from $\cup_{s\in I}\phi_{t(s)}(s))$ flow down and "collide" with $A\cup Q_1$ under the flow, as illustrated in Fig.\ref{disc03} and Fig.\ref{disc04}. Conversely, the other type of domains, $T_{r+1},...,T_k$ satisfies that the trajectories connecting $\cup_{s\in I}\phi_{t(s)}(s)$ to $\partial T_j$, $r+1\leq j\leq k$ do not include boundary arcs on $A$. We note that since $f$ is well-defined, any initial condition in $\cup_{s\in I}\phi_{t(s)}(s)$ eventually flows to $f(s)$, for some $s\in I$. This yields that every initial condition in $\cup_{s\in I}\phi_{t(s)}(s)$ eventually flows to some point in $\partial T_j\cap Q_1$, for some $1\leq j\leq k$.\\

For every $1\leq j\leq k$ there exists $I_j\subseteq I$, a collection of arcs s.t. $f(I_j)=\partial T_j\cap Q_1$ - by its definition, $I=\cup_{j=1}^k \overline{I_j}$. Therefore, for all $1\leq j\leq k$ we now push with the flow the initial conditions in $I_j$ towards $\partial T_j$. When $1\leq j\leq r$, those initial conditions flow until they collide with $\partial T_j\cap Q_1$, as illustrated in Fig.\ref{disc04} (in particular, the trajectories of initial conditions $\partial I_j$ hits $A$ tangently). We now note the trajectories initial conditions on $\partial I_j$, $r+1\leq j\leq k$ spin as they hit $l_1$ at $A$, as illustrated in Fig.\ref{disc03} and Fig.\ref{disc02} for $v_\infty$ and $v_1$. As these trajectories drag between them the flow lines emanating from $I_j$, $r+1\leq j\leq k$, by the orientation preserving properties of the flow we conclude the trajectories of initial conditions on $I_j$ for such $j$ slide below the flow lines emanating from $l_1$ as they flow towards $f(I_j)$, as illustrated in both Fig.\ref{disc03} and Fig.\ref{disc04} (where the flow lines emanating from $l_1$ are illustrated by the purple surface).\\

\begin{figure}[h]
\centering
\begin{overpic}[width=0.6\textwidth]{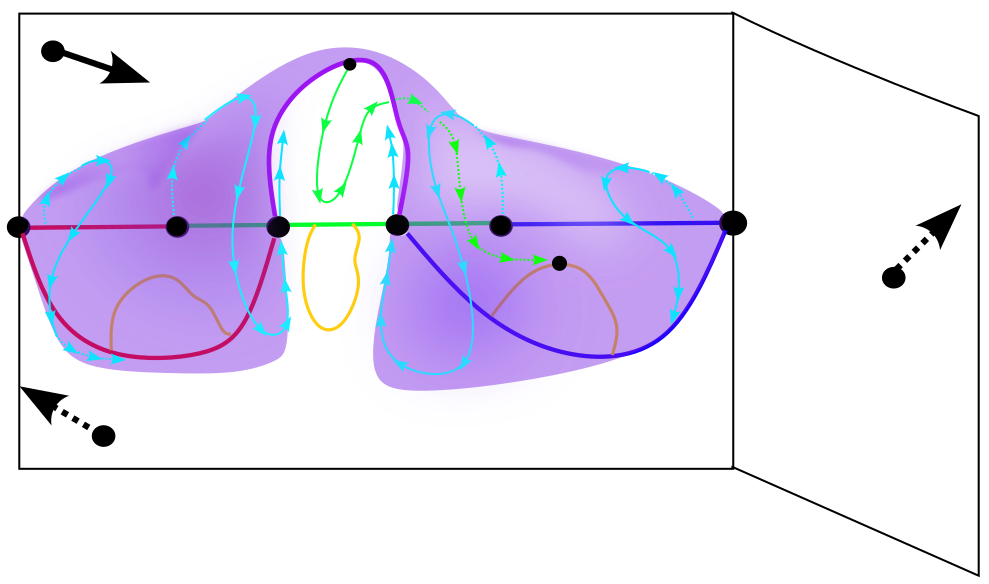}

\put(1000,210){$H_1$}
\put(-40,480){$H_+$}
\put(-40,210){$H_-$}
\put(550,270){$T_3$}
\put(140,270){$T_2$}
\put(140,190){$f(I_\infty)$}
\put(180,385){$v_\infty$}
\put(280,380){$s_\infty$}
\put(620,190){$f(I_1)$}
\put(410,380){$s_1$}
\put(480,380){$v_1$}
\put(350,380){$A$}

\put(330,560){$\mathcal{I}$}
\put(320,310){$T_1$}
\put(630,385){$I_1$}
\put(90,370){$I_\infty$}

\end{overpic}
\caption{\textit{The purple surface is composed of the flow lines connecting $I_1$ to $f(I_1)$, $I_\infty$ to $f(I_\infty)$, and $I$ to $\cup_{s\in I}\phi_{t(s)}(s))$ (the cyan flow lines, denoted by $\mathcal{I}$). Initial conditions on $\mathcal{I}$ that do not tend to $\partial T_1,...\partial T_r$ slide below the purple surface as they tend to $\partial T_{r+1},...,\partial T_k$ (as exemplified by the green flow line beginning in $\mathcal{I}$). In this illustration, $k=3$, $r=1$.}}
\label{disc03}

\end{figure} 

It follows that since $I=\cup_{j=1}^k \overline{I_j}$, for each $r+1\leq j\leq k$ we can define a continuous $t_j:I_j\to\mathbb{R}^+$ be a continuous function defined as described below:
\begin{itemize}
    \item For $s\in I_j$, choose $t_j:\overline{I_j}\to\mathbb{R}^+$ s.t. for all interior $s\in I_j$, $t_j(s)>0$ and $\phi_{t_j(s)}(s)$ is strictly interior to the flow line connecting $s$ to the first intersection point of its trajectory with $Q_1\cup l_1$.
    \item On points $s\in\partial I_j$, if the first intersection point of the trajectory of $s$ with $I$ are at some $s_0\in A$, set $t_j(s)$ s.t. $\phi_{t_j(s)}(s)=(s_0)$. Otherwise, choose $t_j$ s.t. $\phi_{t_j(s)}(s)$ lies strictly between $s$ and the first intersection point of its trajectory with $Q_1\cup l_1$.
\end{itemize}

In other words, we "push down" to $Q_1$ with the flow only the initial conditions on $\cup_{s\in I}\phi_{t(s)}(s)$ that eventually hit the sub-domain on $Q_1$ trapped between $f(I_1),f(I_\infty)$ and $A$ (see the illustration in Fig.\ref{disc04}). Now, write $\partial S_j=(\cup_{s\in \overline{I}_j}\phi_{t_j(s)}(s))\cup A_j$ for some appropriate arc $A_j\subseteq A$ s.t. $\overline{A_j}\cap(\cup_{s\in \overline{I}_j}\phi_{t_j(s)}(s))\ne\emptyset$, where $r+1\leq j\leq k$. By definition, $\partial S_j$ is the boundary of some topological disc $S_j$ in $R_1$, and every initial condition on $\partial S_j$ eventually hits $Q_1$ transversely and escapes the quadrant $R_1$. We now let $\mathcal{T}_j$, $r+1\leq j\leq k$ denote the three dimensional body trapped between $Q_1$, $S_j$, and the flow lines connecting $\partial S_j$ and $Q_1$. It is easy to see $\mathcal{T}_j$ can be visualized as a three dimensional tube, with one opening at $S_j$ where (some) trajectories can (possibly) enter,the tube $\mathcal{T}_j$. Conversely, the regions in $\partial\mathcal{T}_j\cap Q_1$, denoted by $\mu_j$, correspond to where trajectories escape $\mathcal{T}_j$ under the flow.\\

\begin{figure}[h]
\centering
\begin{overpic}[width=0.65\textwidth]{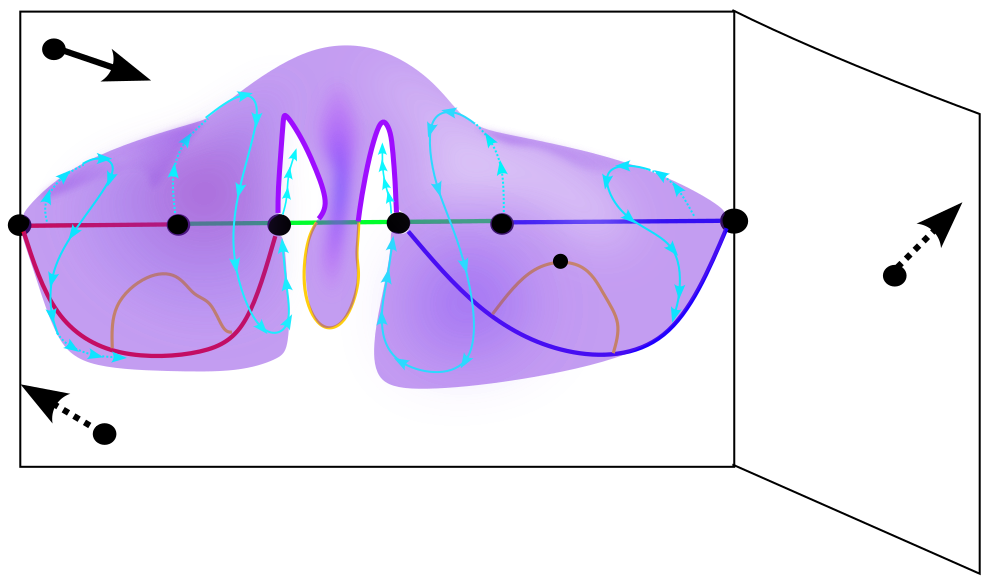}

\put(1000,210){$H_1$}
\put(-40,480){$H_+$}
\put(-40,210){$H_-$}
\put(550,270){$T_3$}
\put(140,270){$T_2$}
\put(140,190){$f(I_\infty)$}
\put(180,385){$v_\infty$}
\put(250,380){$s_\infty$}
\put(620,190){$f(I_1)$}
\put(410,380){$s_1$}
\put(480,380){$v_1$}
\put(360,340){$A_3$}
\put(290,340){$A_2$}
\put(360,490){$\mathcal{I}_3$}
\put(280,500){$\mathcal{I}_2$}
\put(320,310){$T_1$}
\put(630,385){$I_1$}
\put(90,370){$I_\infty$}

\end{overpic}
\caption{\textit{The purple surface is composed of the flow lines connecting $I_1$ to $f(I_1)$, $I_\infty$ to $f(I_\infty)$, $I_j$, $1\leq j\leq r$ to $\partial T_j\cap Q_1$, and $I_j$, $r+1\leq j\leq k$ to $\cup_{s\in \overline{I}_j}\phi_{t_j(s)}(s)$ (in this scenario, $r=1$, $k=3$). Initial conditions on $\mathcal{I}$ that do not tend to $\partial T_1,...\partial T_r$ slide below the purple surface as they tend to $\partial T_{r+1},...,\partial T_k$ (as exemplified by the green flow line beginning in $\mathcal{I}$). In this illustration, $k=3$, $r=1$. $S_j$, $j=2,3$ is the topological disc in $R_1$ trapped between $\mathcal{I}_j$ and the arc $A_j$.}}
\label{disc04}

\end{figure} 

Now, consider $C=R_1\setminus V$, where $R_1$ is as defined before, i.e., the quadrant $\{\dot{x}>0\}\cap\{x<c_1\}$. By definition, for all $r+1\leq j\leq k$, $\mathcal{T}_j$ lies in the same component $\Phi$ of $C$, as the boundary of any $\mathcal{T}_j$ is composed of the openings on $Q_1,S$ and flow lines which emanate from $l_1$ (see the illustration in Fig.\ref{disc03} and Fig.\ref{disc02}). That being said, there exists a component $C_1$ with $x_1$ on its boundary, defined by the flow lines connecting $I_1$ to $f(I_1)$ locally around $x_1$, as illustrated in Fig.\ref{disc03}. The component $C_1$ clearly cannot include a tube $\mathcal{T}_j$, $r<j\leq k$ in it, as by their construction, the interior points in $\mathcal{T}_j$ must be separated from $x_1$ by flow lines on $V$ connecting $S_j$ to the regions $\mu_j$. As such, we conclude $C_1$ does not intersect any of the tubes $\mathcal{T}_j$ - hence it is also disjoint from $\Phi_j$, $r<j\leq k$. As such, this component $C_1$ is trapped between three two dimensional sets:
\begin{itemize}
    \item The flow lines connecting $I_1,I_\infty$ and $\cup_{j=1}^rI_j$ to $f(I_1),f(I_\infty)$ and $\cup_{j=1}^r f(I_j)$.
    \item  The flow lines on $\partial \Phi\cap V$.
    \item The two-dimensional set $Q_1$.
\end{itemize}

Moreover, due to the idealized assumptions on the continuity of $f$ around $x_1$, we conclude locally around $x_1$, $C_1$ has the shape of a topological cone with a tip at $x_1$. All in all, we have proven $C_1$ is a component disjoint from the tubes, and as such, $\partial C_1$ is composed of either regions on $Q_1$ or flow lines on $V$. This concludes the proof of existence of $C_1$ when $v_1\ne v_\infty$.\\

Summarizing our results so far, we have proven that regardless of whether $f$ is continuous or not on $l_1$, $V$ and $Q_1$ trap between them a three-dimensional body whose boundary lies on $V\cup Q_1$. Moreover, we have shown that in all cases, locally around the point $x_1$, $C_1$ has the shape of a topological cone with a tip at $x_1$ (see Fig.\ref{linear}). By this construction of $C_1$ (and by $Q_1=H_-\cup (H_1\cap\{F_1(s)>0\})$), we conclude that given any $s\in\partial C_1$ the point $s$ lies in precisely one of the following three places:

\begin{enumerate}
    \item $V$ - in which case, as $V$ is made of flow lines, the vector $F(s)$ is tangent to $V$ (and hence to $\partial C_1$).
    \item $H_-$ - in which case $F(s)$ points into the region $\{\dot{x}(s)<0\}$.
    \item $H_1\cap\{\dot{x}(s)>0\}$ - in which case $F(s)$ points into the region $\{(x,y,z)\mid x>c_1\}$.
\end{enumerate}

Since by construction $C_1\subseteq\{\dot{x}(s)\geq0\}\cup\{(x,y,z)\mid x\leq c_1\}$ it follows that throughout $\partial C_1$ the vector field $F$ is either tangent to $\partial C_1$ or points outside of it - in particular, no initial condition can enter $C_1$ under the flow. Now, recall that per the idealized assumptions on $F$ the Jacobian matrix at $x_1$ has no purely imaginary eigenvalues, and denote the said Jacobian by $J_1$. From the Hartman-Grobman Theorem it follows that under the idealized assumptions the local dynamics of $\dot{s}=F(s)$ around $x_1$ are orbitally equivalent to those generated by $\dot{s}=J_1s$ around the origin. In other words, there exists two open balls, $B_r(x_1)$ and $B_1(0)$ and a homeomorphism $h:B_r(x_1)\to B_1(0)$ which takes the flow lines of $F$ in $B_r(x_1)$ to those of $J_1$ in $B_1(0)$. As such, by the discussion above we conclude $h(C_1\cap B_r(x_1))=K_1$ is also a cone with a tip at $0$ into which nothing can enter (see Fig.\ref{linear}). This immediately implies $K_1$ must include some eigenvector for $J_1$. As a consequence,  $C_1$ intersects some invariant manifold $\Gamma_1$ for $x_1$.\\ 

We now claim $\Gamma_1$ is unbounded. To see why, consider some $s\in\Gamma_1$ and recall that for every $s\in\overline{C_1}$, the $x$ coordinate of $s$ is at most $x_1$. Now, further note that by construction no trajectory can enter $C_1$ under the flow generated by $F$, which implies no trajectory can escape $\overline{C_1}$ under the inverse flow, generated by $-F$. Recalling we denote the flow function by $\phi_t,t\in\mathbb{R}$, this shows that for every $s\in\Gamma_1$ and every $t<0$ we have $\phi_t(s)\in\overline{C_1}$. Moreover, by $C_1\subseteq\{\dot{x}(s)\geq0\}\cap\{(x,y,z)\mid x\leq c_1\}$ it follows the $x$--velocity on the backwards trajectory of $s$ is always non-negative - and since $x_1$ is the only fixed point in $\overline{C_1}$ this implies that under the inverse flow the $x$ coordinate of the point $\phi_t(s)$ must tend to $-\infty$ as $t\to-\infty$. Or, in other words, we have just proven $\Gamma_1$ is trapped in $\overline{C_1}$ and connects $x_1$ to $\infty$. This concludes the proof of existence of $\Gamma_1$ under the idealized assumptions.\\

\begin{figure}[h]
\centering
\begin{overpic}[width=0.35\textwidth]{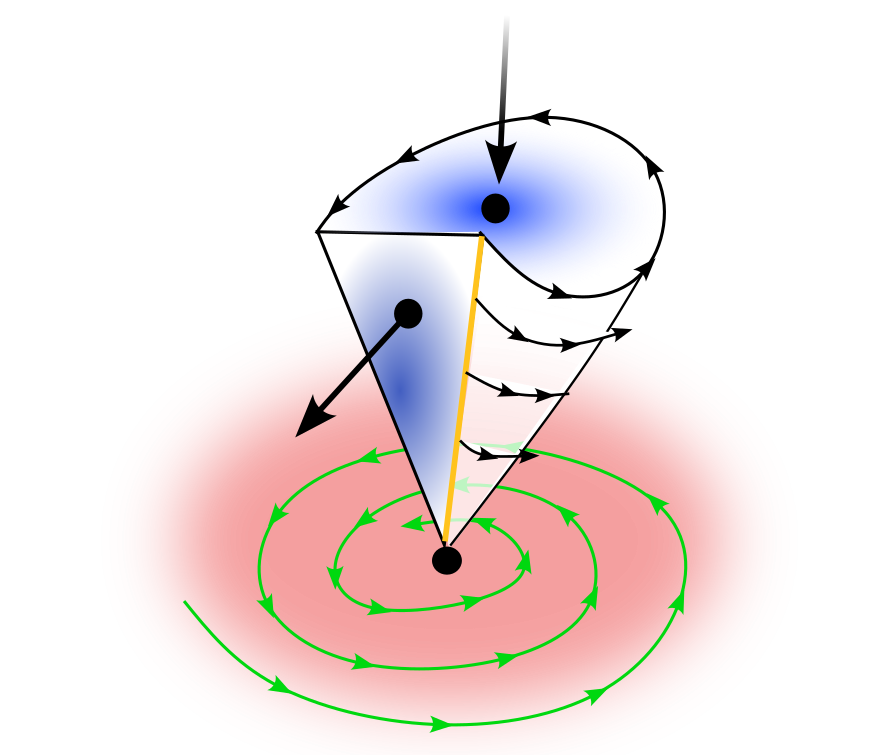}
\end{overpic}
\caption{\textit{The topological cone with a tip at $x_1$.}}\label{linear}

\end{figure}
\subsection{Stage III - removing the idealized conditions.}
Having proven the existence of $\Gamma_1$ under idealized assumptions, we now do the same in the general case - i.e., we now show these assumptions can be removed by resorting to a method of approximation. To do so, assume the vector field $F$ does not satisfy the idealized assumptions. We first note the function $f$ is still defined, even if discontinuous around $\infty$ and $x_1$ (this follows since any initial condition on $l_1$ has to hit either $H_1$ or $H_-$ in finite time). We now smoothly deform it as follows:

\begin{itemize}
    \item We smoothly deform the flow generated by $F$ in some small neighborhood of $x_1$ (if necessary) by smoothly deforming the Jacobian matrix of $x_1$, $J_1$, such that it has no imaginary eigenvalues, and $f$ is continuous around $x_1$. Note that this does not change the Poincar\'e Index of $x_1$.
    \item We use the assumption that the Poincaré index of $F$ at $\infty$ is $d$ (where $d\in\{0,\pm1\}$). We first consider only the cases where $d=0,1$. To begin, sing Hopf's Theorem we homotopically deform the dynamics of $F$ on some arbitrarily small neighborhood $N$ of $\infty$ such that $\infty$ becomes a smooth fixed point of Poincaré index $-d$. This allows us to deform $F$ in some arbitrarily small neighborhood of $\infty$ to ensure the following:

    \begin{enumerate}
        \item When $d=0$, we deform the behavior around $\infty$ s.t. the flow behaves as in Fig.\ref{ind}. In particular, we do so s.t. the trajectories of initial conditions on $l_1$ sufficiently close to $\infty$ hit $Q_1$ transversely as shown in the lower image in Fig.\ref{ind1} - i.e., we ensure both the continuity of $f$ around $\infty$, and $\lim_{s\to\infty}f(s)=\infty$ (w.r.t. the deformation).
        \item When $d=1$, we make $\infty$ a saddle focus of Poincar\'e Index $-1$. Again, this allows us to deform the flow around $\infty$ s.t. the trajectories of initial conditions hit $Q_1$ transversely, thus again ensuring the continuity of $f$ around $\infty$, and $\lim_{s\to\infty}f(s)=\infty$ (w.r.t. the deformation). See the illustration in the upper left image in Fig.\ref{ind1}.
        \item When $d=-1$ we make $\infty$ a source with complex-conjugate eigenvalues (it will trivially be of Poincar\'e Index $1$). Again, this allows us to ensure that w.r.t. this deformation $f$ is continuous on $l_1$ around $\infty$, and satisfies $\lim_{s\to\infty}f(s)=\infty$ (see the illustration in the upper right image in Fig.\ref{ind1}).
    \end{enumerate}
\end{itemize}

\begin{figure}[h]
\centering
\begin{overpic}[width=0.5\textwidth]{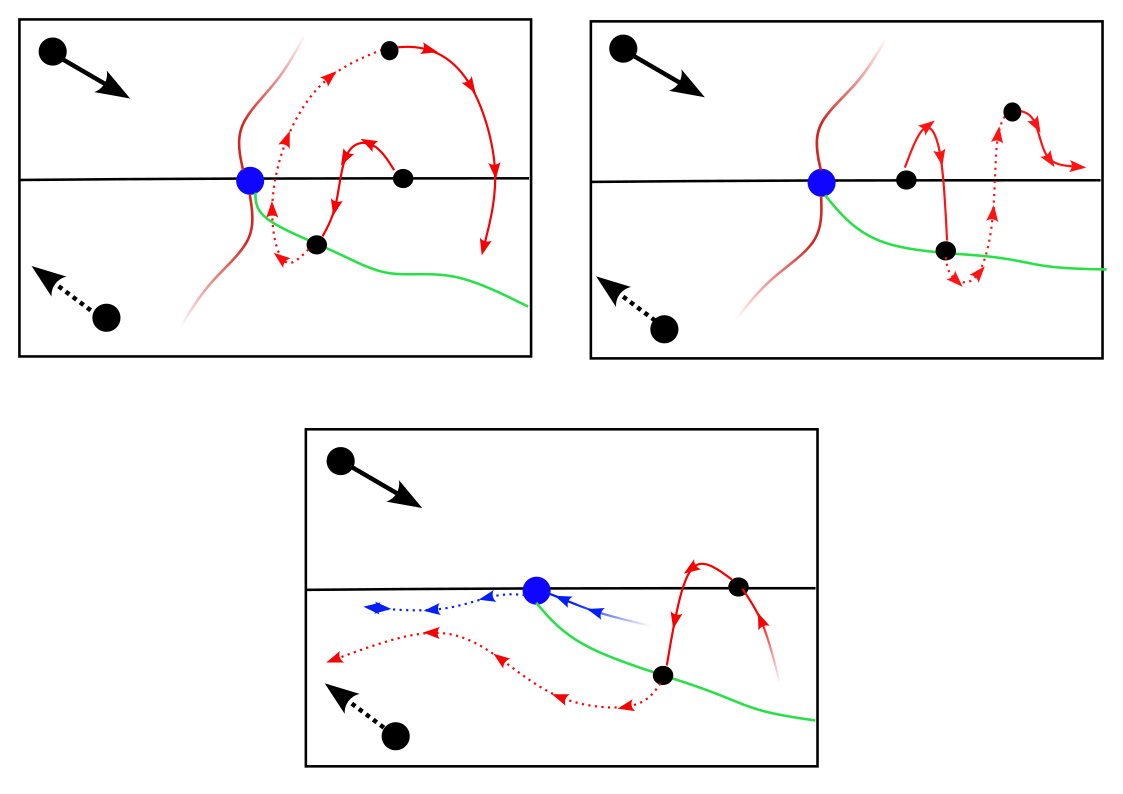}
\put(100,470){$H_-$}
\put(570,470){$H_-$}
\put(390,50){$H_-$}

\end{overpic}
\caption{\textit{The local dynamics around $\infty$ after the deformation, when $d=1,-1,0$. In all figures the blue dot denotes $\infty$, while the square denotes $H$. The green curve always denotes the image of $l_1$ under $f$, while the red flow line denotes the trajectory of an initial condition on $l_1$. The dark arrows denote the directions of the vector field on the two half planes. On the upper left image we illustrate the scenario around $\infty$ when $d=1$ - in this scenario $\infty$ (the blue dot) is a saddle focus with a two-dimensional stable manifold (the brown line), causing the trajectories of initial conditions on $l_1$ to spiral around it. On the upper right figure we have the scenario where $d=-1$, where $\infty$ is a source with a two-dimensional invariant manifold corresponding to two complex conjugate eigenvalues, around which the trajectories of initial conditions in $l_1$ also spiral. The lower image denotes the scenario where $d=0$, in which $\infty$ is a fixed point of Poincar\'e Index $0$ with invariant manifolds flowing towards and away from it.}}\label{ind1}

\end{figure}
At this point we remark the deformation around $\infty$ depends on the Poincar\'e Index at $\infty$ alone. In addition, let us also recall that since the deformation around $x_1$ cannot change the Poincar\'e Index, since $x_1$ is non-degenerate and isolated. As such, using the exact same argument previously applied to $\infty$ and now applied to $x_1$, we deform the local dynamics around $x_1$ s.t. $f$ is continuous there, and satisfies $\lim_{s\to x_1}f(s)=x_1$. Combined with the deformation above that already ensures $\lim_{s\to\infty}f(s)=\infty$ and the continuity of $f$ around $\infty$, we have proven $F$ can be deformed away from $x_1,\infty$ s.t. the new vector field satisfies the assumptions of Stage $II$. Moreover, that new vector field can be chosen to coincide with the original one, $F$, on some arbitrarily large set in $S^3$.\\

Let us denote this new vector field by $F'$. As mentioned above, we choose$F'$ such that it coincides with $F$, the original vector field, on an arbitrarily large set of $\mathbb{R}^3$. It is therefore immediate $F'$ satisfies the assumptions of the idealized scenario studied in Stage $II$ - which implies that with respect to $F'$ the fixed point $x_1$ generates a one-dimensional invariant manifold $\Gamma'_1$ which satisfies the following two assertions:

\begin{enumerate}
    \item $\Gamma'_1\subseteq\{\dot{x}\geq0\}\cap\{(x,y,z)\mid x<c_1\}$.
    \item  $\Gamma'_1$ connects $x_1$ and $\infty$. 
\end{enumerate}

where the velocities above are taken with respect to the new vector field $F'$. Now, let $D$ denote the subset of $S^3$ on which $F'$ and $F$ differs. As remarked above, we can choose $D$ to be arbitrarily small (in $S^3$) - which implies that since for every $F'$ the invariant manifold $\Gamma'_1$ as above exists and satisfies the above properties, there must also exist $\Gamma_1$, an invariant manifold for $x_1$ connecting $x_1$ and $\infty$, which is trapped in $\{\dot{x}\geq0\}\cap\{(x,y,z)\mid x<c_1\}$ (this follows as trivially, $F'\to F$ in the $C^k$ metric in $\mathbb{R}^3$ away from some neighborhoods of $x_1,\infty$). This concludes the proof of Case $A$.\\

We now sketch the proof of Case $B$, which, as remarked earlier, is very similar to Case $A$ - recalling the definition of $t(s)$ in \pageref{time}, we recall that in Case $B$ the scenario is that for all $s\in l_1$ the curve $\phi_{(-t(s),t(s))}(s)$ is in $\{F_1(s)\leq0\}$. In particular, in Case $B$ the backwards trajectory of every $s\in l_1$ either hits transversely $H_-$ and enters $\{\dot{x}(s)>0\}$ in backwards time, or it hits transversely $H_1\cap \{\dot{x}(s)<0\}$ and enters $\{x>c_1\}\cap \{\dot{x}(s)<0\}$ in backwards time. This implies that analogously to $p(s)$ in Case $A$, in Case $B$ we can define $q(s)$ for any given $s\in l_1$ such that $q(s)$ is the first backwards time for which $\phi_{q(s)}(s)\in H_+\cup(H_+\cap\{\dot{x}(s)<0\})$. By replacing the vector field $F$ with $-F$ similar arguments to those above imply the same conclusion follows for Case $B$ as well - i.e., similar arguments now imply that Case $B$ also forces the existence of some one-dimensional invariant manifold $\Gamma_1$ on $x_1$ such that $\Gamma_1\subseteq\{\dot{x}\leq0\}\cap\{(x,y,z)\mid x<c_1\}$.\\

\begin{figure}[h]
\centering
\begin{overpic}[width=0.5\textwidth]{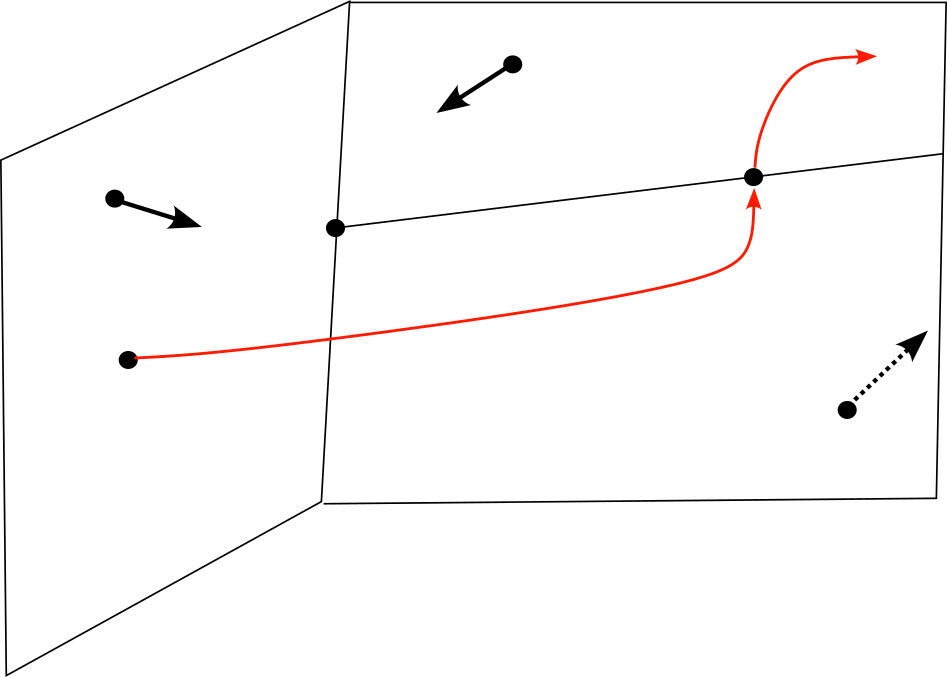}
\put(100,360){$s$}
\put(600,590){$H_+$}
\put(700,325){$H_-$}
\put(180,150){$H_2$}
\put(290,450){$x_2$}
\put(500,440){$l_2$}
\end{overpic}
\caption{\textit{The geometric configuration for $x_2$ with the directions of the vector field on $H_2,H_+$ and $H_-$ - as every initial condition on $l_2$ flows in backwards time towards either $H_2$ or $H_+$, similar arguments to those used before imply the existence of an unbounded invariant manifold for $x_2$.}}\label{comp4}

\end{figure}

\subsection{Stage IV - concluding the proof.}

Having proven the existence of $x_1$ and $\Gamma_1$ in both Cases $A$ and $B$, we now sketch the proof for the existence of the fixed point $x_2$ and the corresponding invariant manifold $\Gamma_2$. To begin, recall we defined $x_2$ as the fixed point which maximizes the $x$--coordinate on $l$ - note $x_2$ may or may not be distinct from $x_1$, as illustrated in Fig.\ref{comp} (similar arguments to those used to prove the existence of $x_1$ also imply its existence). Similarly to the previous case, it would suffice to prove the existence of $\Gamma_2$ under the non-genericity assumption for $F$ - i.e., it would suffice to prove it under the assumption there is no sub-arc $\delta\subseteq l$, $\delta\ne\Gamma_1$, connecting some fixed point $w'$ to $\infty$ - for again, if there exists such a curve we just set $\Gamma_2=\delta$ and $x_2=w'$.\\

To begin, set $x_2=(a_1,a_2,a_3)$, and consider the plane $H_2=\{(a_1,y,z)\mid y,z\in\mathbb{R}\}$ - which will play an analogous role to that of $H_1$. Per our assumptions on $F$, we know $H_2$ intersects the cross-section $H$ as appears in Fig.\ref{comp4} - and analogously to $l_1$, we define $l_2$ to be the sub-arc of $l$ connecting $x_2$ and $\infty$ through $\{(x,y,z)\mid x>a_1\}$. By studying the trajectories of initial conditions $s\in l_2$ similar arguments to those used above imply the existence of an invariant manifold $\Gamma_2$ for $x_2$, such that the following holds:

\begin{enumerate}
    \item $\Gamma_2$ is a curve connecting $\infty$ and $x_2$.
    \item $\Gamma_2\subseteq\{(x,y,z)\mid x>a_1\}$.
    \item The $\dot{x}$ velocity on $\Gamma_2$ never changes its sign.
\end{enumerate}

Having proven the existence of both $\Gamma_1$ and $\Gamma_2$ we are now in a position to conclude the proof of Theorem \ref{infinitytheorem}. To do so, all that remains is to show the invariant manifolds $\Gamma_1$ and $\Gamma_2$ are not knotted with themselves - in the sense that there exists a curve $\gamma$ connecting $x_1$ and $x_2$ such that the union $\{x_1,x_2,\infty\}\cup\Gamma_1\cup\Gamma_2\cup\gamma$ is the unknot. To do so, first note that by construction we have $\Gamma_1\subseteq\{(x,y,z)\mid x<c_1\}$ and $\Gamma_2\subseteq\{(x,y,z)\mid x>a_1\}$ - by $a_1\geq c_1$ it follows the curves $T_1=\Gamma_1\cup\overline{l_1}$ and $T_2=\Gamma_2\cup\overline{l_2}$ cannot be linked with one another (where the closure $\overline{l_i}, i=1,2$ is taken in $S^3$). Similarly, as the $\dot{x}$ velocity on either $\Gamma_1$ or $\Gamma_2$ never changes its sign, it follows that whenever either one of the curves $T_1$ and $T_2$ are knots, their knot-type can only be the $S^1$, i.e., both are unknots. As such, setting $\gamma$ as the straight line connecting $x_1$ and $x_2$ we see $\{\infty,x_1,x_2\}\cup\gamma\cup\Gamma_1\cup\Gamma_2$ can only be the unknot The proof of Theorem \ref{infinitytheorem} is now complete.
\end{proof}
\begin{remark}
    The assumptions of Theorem \ref{infinitytheorem} can be modified to derive similar results for vector fields that do not fit precisely into the assumptions of Theorem \ref{infinitytheorem}. We given an example of such a case in Section \ref{nonstand}.
\end{remark}
\begin{remark}
\label{nocon}    It is easy to see from the proof of Theorem \ref{infinitytheorem} that $x_1=x_2$ precisely when $F$ generates one fixed point. Or in other words, whenever $F$ has more than one fixed point in $\mathbb{R}^3$, we have $x_1\ne x_2$.
\end{remark}

\section{The applications:}
Having proven Theorem \ref{infinitytheorem}, in this section we show how to apply it. In more detail, in this section we apply Theorem \ref{infinitytheorem} to study three examples of three-dimensional flows - the Belousov--Zhabotinsky reaction, the Genesio--Tesi system, and the Michelson system (see \cite{BZ}, \cite{GT} and \cite{Michh}, respectively).\\

This section is organized as follows - in Section \ref{ap1} we apply Theorem \ref{infinitytheorem} directly to prove the existence of unbounded invariant manifolds for both the Belousov--Zhabotinsky reaction and the Genesio--Tesi system. Following that, in Section \ref{ap2} we apply Theorem \ref{infinitytheorem} to study the topology generated by the Michelson system - in particular, by combining Theorem \ref{infinitytheorem} and the results of \cite{DW} we prove the Michelson system generates infinitely many homoclinic nooses and heteroclinic knots (see Definition \ref{knotnoose}). In particular, our results on the Michelson system in Section \ref{ap2} exemplify the potential uses of Theorem \ref{infinitytheorem} to the study of forcing phenomena for three-dimensional flows.\\

Before we begin, we remark certain ideas from the proof of Theorem \ref{infinitytheorem} can also be applied to study systems which do not necessarily satisfy all the assumptions of Theorem \ref{infinitytheorem} - or in other words, Theorem \ref{infinitytheorem} can probably be generalized, possibly in more than one way. For completeness' sake, we defer the discussion in such possible generalizations of Theorem \ref{infinitytheorem} to Section \ref{nonstand}, where we show how this can be done via a concrete example. 

\subsection{Unbounded dynamics in the Belousov--Zhabotinsky reaction and the Genesio--Tesi system.}
\label{ap1}
In this section we apply Theorem \ref{infinitytheorem} to study two classical examples of three-dimensional flows - the Belousov--Zhabotinsky and the Genesio--Tesi system. We begin with the Belousov--Zhabotinsky equation - to do so, given three positive parameters $a,b,c\in\mathbb{R}$ we define the Belousov--Zhabotinsky reaction as follows (see \cite{BZ}):

\begin{equation} \label{BZ}
\begin{cases}
\dot{x} = y \\
 \dot{y} = z\\
 \dot{z}=x-4 y-z+x^2-ay^2-bxz-cx^2z
\end{cases}
\end{equation}

As observed numerically, there exists parameter values for which the flow generates a chaotic invariant set - see \cite{BZ} for the details. It is easy to see the flow generated by the equations above has precisely two fixed points - $(0,0,0)$, a saddle focus of Poincaré index $1$, and $(-1,0,0)$, a sink with two complex-conjugate eigenvalues of Poincaré index $-1$. Moreover, by computation the divergence of this system is negative precisely in the region $\{(x,y,z)\mid -1-bx-cx^2<0\}$, and positive otherwise - or in other words, heuristically one should not expect the flow to have a global attractor.\\

Now, consider the plane $H=\{\dot{x}=0\}=\{(x,0,z)\mid x,z\in\mathbb{R}\}$ - as the normal vector to $H$ is $(0,1,0)$ by direct computation we see the vector field is tangent to $H$ precisely on the curve $l=\{(x,0,0)\mid x\in\mathbb{R}\}$. Similarly, we also have the equalities $\{\dot{x}>0\}=\{(x,y,z)\mid y>0\}$ and $\{\dot{x}<0\}=\{(x,y,z)\mid y<0\}$ (see Fig.\ref{Bpic}) - and finally, it is easy to see $H$ is transverse to any plane parameterized by $\{(c,y,z)\mid y,z\in\mathbb{R}\}, c\in\mathbb{R}$. We now evaluate the behavior of the flow on the curve $l=\{(x,0,0)\mid x\in\mathbb{R}\}$ - noting that on any point $s=(x,0,0)$ the vector field points in the direction $(0,0,x(x+1))$ we conclude the following:

\begin{itemize}
    \item For all $x\in(-1,0)$, the vector field points at $s$ in the negative $z$--direction (see Fig.\ref{Bpic}).
    \item For $x\ne[-1,0]$ the vector field at $s$ points in the positive $z$--direction (see Fig.\ref{Bpic}).
\end{itemize}

To continue, consider the set $H_+=\{(x,0,z)\mid z>0\}$ and $H_-=\{(x,0,z)\mid z<0\}$. By direct computation one sees that on $H_+$ the flow trajectories cross from $\{\dot{x}\leq0\}$ into $\{\dot{x}>0\}$ - and similar arguments prove that on $H_-$ the flow trajectories cross from $\{\dot{x}\geq0\}$ into $\{\dot{x}<0\}$ (see Fig.\ref{Bpic}). Combining this information with the behavior of the vector field on the curve $l$ (as discussed above), we immediately conclude:

\begin{itemize}
    \item  For $x\in(-1,0)$ the trajectory of $s\in l$, $s=(x,0,0)$ arrives at $s$ from $\{\dot{x}<0\}$ and re-enters $\{\dot{x}<0\}$ immediately upon leaving $s$ (see Fig.\ref{Bpic}).
    \item Similarly, for $x\not\in[-1,0]$ the trajectory of $s\in l$, $s=(x,0,0)$ arrives at $s$ from $\{\dot{x}>0\}$, hits $s$, and immediately re-enters $\{\dot{x}>0\}$ upon leaving $s$ (see Fig.\ref{Bpic}).
\end{itemize}

Or put simply, we have just shown the sets $l$, $H$, $H_\pm$ and the planes $H_c=\{(c,y,z)\mid y,z\in\mathbb{R}\}$ satisfy the assumptions of Theorem \ref{infinitytheorem}. As we have already shown the fixed points $(0,0,0)$ and $(1,0,0)$ each have a pair of complex-conjugate eigenvalues - as the first is a saddle focus and the second a complex sink - it follows that to prove the Belousov--Zhabotinsky reaction satisfies the assumptions of Theorem \ref{infinitytheorem} it remains to compute the Poincaré index at $\infty$.\\

To do so, for simplicity, denote by $F$ the vector field generating the Belousov--Zhabotinsky reaction as defined above - we now prove the Poincaré index of $F$ at $\infty$ is $0$. To do so, we first note that by direct computation one obtains the equality $l=\{\dot{y}=0\}\cap\{\dot{x}=0\}=\{(x,0,0)\mid x\in\mathbb{R}\}$ - which implies $F$ can point in the $(0,0,-1)$ direction only on $s\in l$. Now, choose some large sphere $S_r=\{\mid \mid s\mid \mid =r\}$ and note $l\cap S_r$ consists of precisely two points - $(-r,0,0),(r,0,0)$ and that $F(\pm r,0,0)=(0,0,\pm r(\pm r+1))$. When $\mid r\mid $ is sufficiently large, $F(\pm r,0,0)$ points in the positive $z$--direction - i.e., the map $\frac{F}{\mid \mid F\mid \mid }:S_r\to S^2$ does not include the direction $(0,0,-1)$ in its image. This immediately implies the degree of $\frac{F}{\mid \mid F\mid \mid }$ on $S_r$ is $0$ - and since $r>1$ was chosen arbitrarily, it follows the Poincaré index at $\infty$ is also $0$. Therefore, all in all, it follows the Belousov--Zhabotinsky reaction satisfies all the assumptions of Theorem \ref{infinitytheorem}. Denoting $x_1=(0,0,0)$ and $x_2=(-1,0,0)$ we conclude by Remark \ref{nocon} that each one of these fixed points generates an unbounded invariant manifold, $\Gamma_1$ and $\Gamma_2$, connecting it to $\infty$ (see Fig.\ref{Bpic}).\\
\begin{figure}[h]
\centering
\begin{overpic}[width=0.5\textwidth]{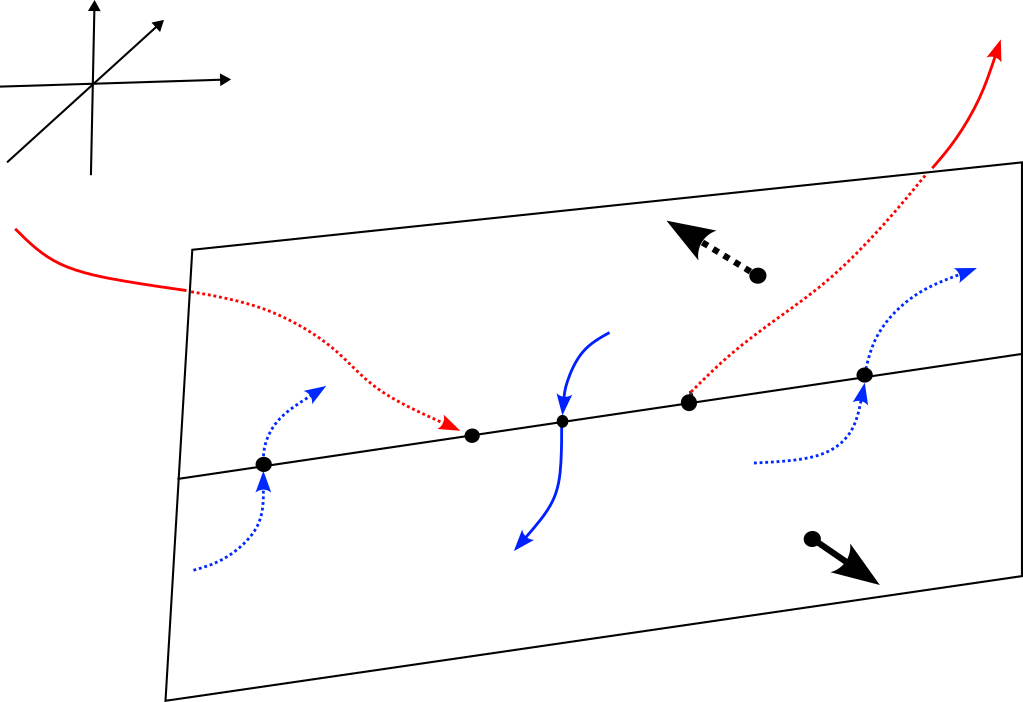}
\put(590,550){$\{\dot{x}>0\}$}
\put(360,210){$(-1,0,0)$}
\put(1030,350){$\{\dot{x}<0\}$}
\put(670,250){$(0,0,0)$}
\put(720,450){$H_+$}
\put(700,120){$H_-$}
\put(230,595){$x$}
\put(165,680){$y$}
\put(75,700){$z$}
\end{overpic}
\caption{\textit{The configuration of $H=H_+\cup H_-$ and $l$ for the Belousov--Zhabotinsky reaction (along with the directions of the vector field on them). The red curves denote the unbounded invariant manifolds given by Theorem \ref{infinitytheorem}, while the blue flow lines denote the (tangent) local dynamics on $l$.}}\label{Bpic}

\end{figure}

Having analyzed the Belousov--Zhabotinsky reaction, we now briefly apply similar ideas to study the Genesio--Tesi system - as the arguments are very similar to those used to analyze the Belousov--Zhabotinsky reaction we only give a rough sketch of the proof, avoiding the technical details. To begin, given constants $a,b>0$ we define the Genesio--Tesi equations (see \cite{GT}) by the following vector field:

\begin{equation}
\begin{cases}
\dot{x} = y \\
 \dot{y} = z\\
 \dot{z}=-az-by-x(1+x)
\end{cases}
\end{equation}

As observed numerically, there are parameter values at which there exists a chaotic invariant set - see \cite{GT} for the details. By direct computations, the fixed points are given by $(0,0,0),(-1,0,0)$ - and moreover, one can show both are saddle foci of opposing Poincaré indices. Denoting the vector field by $F$, similar arguments to those used to analyze the Belousov--Zhabotinsky equations now show the Genesio--Tesi system also satisfies the following:

\begin{enumerate}
    \item The Poincaré index at $\infty$ is $0$.
    \item The set $H=\{\dot{x}=0\}$ is a plane, and the tangency set for it is $l=\{(x,0,0)\mid x\in\mathbb{R}\}$.
    \item $H\setminus l$ is composed of two half-planes: the upper $H_+$ at which trajectories cross from $\{\dot{x}\leq0\}$ to $\{\dot{x}>0\}$, and $H_-$ where the opposite occurs (see Fig.\ref{Bpic}).
    \item For all $x>0$ and $x<-1$ the vector $F(x,0,0)=(0,0,-x(1+x))$ points in the negative $z$--direction. As a consequence, for such an $x$ the flow line arrives at $s=(x,0,0)$ from $\{\dot{x}<0\}$, hits $s$, and returns back to $\{\dot{x}<0\}$.
\end{enumerate}

The geometric scenario here is almost the same as that of the Belousov--Zhabotinsky reaction, as sketched in Fig.\ref{Bpic} - and as such, the Genesio--Tesi system also satisfies the assumptions of Theorem \ref{infinitytheorem}. Consequentially, by Theorem \ref{infinitytheorem} and Remark \ref{nocon} we conclude both fixed points $(-1,0,0)$ and $(0,0,0)$ are connected by one-dimensional invariant manifolds to $\infty$.
\subsection{Homoclinic nooses and Heteroclinic knots in the Michelson System.}
\label{ap2}
Given any $c>0$ we define the Michelson System as the flow generated by the following system of differential equations:

\begin{equation} \label{Mich}
\begin{cases}
\dot{x} = y\\
 \dot{y} = z\\
 \dot{z}=c^2-y-\frac{x^2}{2}
\end{cases}
\end{equation}
It is easy to prove this system has precisely two fixed points $p_+=(c\sqrt{2},0,0)$ and $p_-=(-c\sqrt{2},0,0)$ - it is also easy to prove the divergence of this dynamical system vanishes throughout $\mathbb{R}^3$. In this section we use Theorem \ref{infinitytheorem} to prove there exist infinitely many $c>0$ where the Michelson system generates an invariant one-dimensional structure - i.e., we prove there exist infinitely many $c>0$ in which the Michelson system generates either a heteroclinic knot or a homoclinic noose (see Definition \ref{knotnoose} and Theorem \ref{noose}).\\

We begin with basic qualitative analysis of the Michelson system. To do so, we first recall the Shilnikov condition (see \cite{LeS}) - i.e., given a saddle focus type fixed point $p$ with eigenvalues $\gamma,\rho\pm i\omega$ (where $\rho$ and $\gamma$ have opposite signs), we say $p$ satisfies the \textbf{Shilnikov Condition} if $\mid \frac{\rho}{\gamma}\mid <1$. As proven in \cite{LeS}, whenever a fixed points both satisfies the Shilnikov condition and generates a homoclinic trajectory, the local dynamics around it will include infinitely many suspended Smale Horseshoes. With these ideas in mind, we now prove the following technical Lemma:
\begin{lemma}
    For all $c>0$, $p_+$ and $p_-$ are saddle foci, of opposing indices. Moreover, both satisfy the Shilnikov condition.
\end{lemma}
\begin{proof}
    We first remark the vector field generating Eq.\ref{Mich} is invariant under the symmetry $(x,y,z,t)\to(-x,y,-z,-t)$, which is seen by direct computation (where $t\in\mathbb{R}$ is the time variable). This implies it will suffice to prove $p_+$ is a saddle focus of Poincaré index $-1$ satisfying the Shilnikov condition - by the symmetry of the vector field it would immediately follow the same is true for $p_-$, and that it is a saddle focus of Poincaré index $1$. To begin, set $J_+$ as the Jacobian matrix of the vector field at $p_+$ - by computation, $J_+$ has a negative determinant, one real eigenvalue, and two eigenvalues given by the following formulas:

    \begin{equation} \label{eigen1}
\frac{1+i\sqrt{3}}{2^{\frac{2}{3}}3^{\frac{1}{3}}(\sqrt{3}\sqrt{27(c\sqrt{2})^2+4}-9c\sqrt{2})^{\frac{1}{3}}}-\frac{(1- i\sqrt{3})(\sqrt{3}\sqrt{27(c\sqrt{2})^2+4}-9c\sqrt{2})^{\frac{1}{3}}}{2^{\frac{4}{3}}3^{\frac{2}{3}}}
\end{equation}
  \begin{equation} \label{eigen2}
\frac{1-i\sqrt{3}}{2^{\frac{2}{3}}3^{\frac{1}{3}}(\sqrt{3}\sqrt{27(c\sqrt{2})^2+4}-9c\sqrt{2})^{\frac{1}{3}}}-\frac{(1+ i\sqrt{3})(\sqrt{3}\sqrt{27(c\sqrt{2})^2+4}-9c\sqrt{2})^{\frac{1}{3}}}{2^{\frac{4}{3}}3^{\frac{2}{3}}}
\end{equation}

As $c>0$ we have $\sqrt{27(c\sqrt{2})+4}>-9c\sqrt{2}$ - which implies these eigenvalues are complex-conjugate. In addition, by direct computation one also sees the real eigenvalue is given by the formula:

 \begin{equation} \label{eigen3}
\frac{(\sqrt{3}\sqrt{27(c\sqrt{2})^2+4}-9c\sqrt{2})^{\frac{1}{3}}}{2^{\frac{1}{3}}3^\frac{2}{3}}-\frac{2^{\frac{1}{3}}3^\frac{2}{3}}{(\sqrt{3}\sqrt{27(c\sqrt{2})^2+4}-9c\sqrt{2})^{\frac{1}{3}}}
\end{equation}
It follows the real part of the complex-conjugate eigenvalues and the real eigenvalues have opposite signs - combined with $\det(J_+)<0$ we conclude the real eigenvalue is negative. This shows $p_+$ is a saddle focus of Poincaré index $-1$ - moreover, by direct computation it easily follows $p_+$ satisfies the Shilnikov condition and the conclusion follows.
\end{proof}
To continue, we now introduce the following definitions:
\begin{definition}
\label{knotnoose}    Let $F$ be a smooth vector field of $\mathbb{R}^3$, with a finite number of fixed points $x_1,\ldots ,x_n$, with one-dimensional invariant manifolds $W_1,\ldots ,W_n$. We say $F$ generates a \textbf{homoclinic noose} (see Fig.\ref{knotnosepic}) if every fixed point satisfies the following:
    \begin{itemize}
        \item For all $i$, $W_i$ includes a bounded homoclinic trajectory, $\Gamma_i$.
        \item For all $i$, $W_i\setminus\Gamma_i$ is an unbounded set, i.e., an invariant manifold connecting $x_i$ and $\infty$.
    \end{itemize}
    Similarly, we say $F$ generates a \textbf{heteroclinic knot} provided the following is satisfied (see Fig.\ref{knotnosepic}):
    \begin{itemize}
        \item Every component in $W_i$, $1\leq i\leq n$ connects $x_i$ to either another fixed point $x_j$ or to $\infty$,
        \item $\{x_1,\ldots ,x_n,\infty\}\cup W_1\cup\ldots \cup W_n$ is a knot (in $S^3$).
    \end{itemize}
\end{definition}
\begin{figure}[h]
\centering
\begin{overpic}[width=0.35\textwidth]{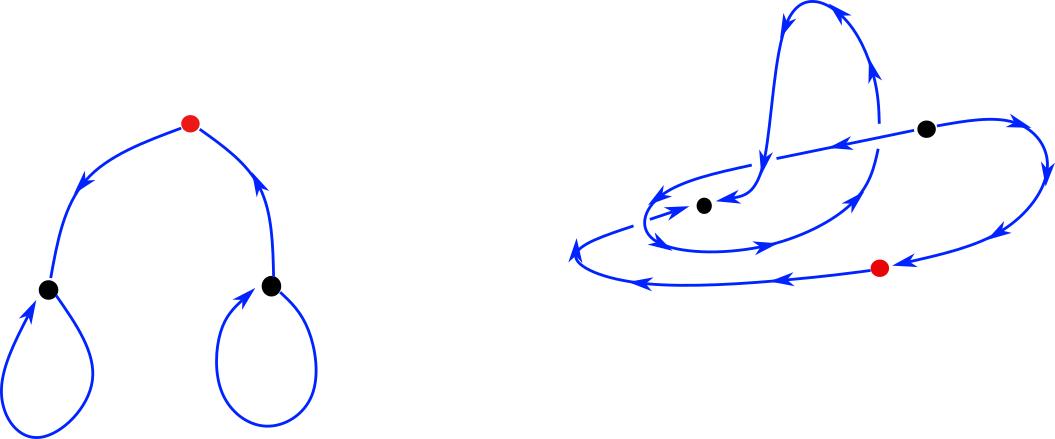}
\end{overpic}
\caption{\textit{A homoclinic noose on the left and a heteroclinic knot on the right. The black dots always denote finite fixed points, while the red dot denotes $\infty$.}}\label{knotnosepic}

\end{figure}

Similarly to the homoclinic Shilnikov scenario, Heteroclinic knots are well-known to be connected with the onset of chaos in three-dimensional flows - see for example \cite{Pi}, where the notion of a heteroclinic knot is used to analytically prove the existence of chaotic dynamics in the Lorenz system. The key heuristic behind the connection between heteroclinic dynamics and chaos is that the behavior of a given vector field $F$ on a heteroclinic knot $H$ can force to the dynamics in $\mathbb{R}^3\setminus H$ to behave in a certain way - or in other words, the topology of $\mathbb{R}^3\setminus H$ and the behavior of $F$ on $H$ can force complex dynamics to appear. As such, one would expect that given any one-dimensional set $H$ which is either a homoclinic noose or a heteroclinic knot the topology of $H$ could potentially force complex dynamics to appear.\\

We will not test out these ideas on the Michelson system, as the study of which homoclinic nooses and heteroclinic knots force complex dynamics to appear is well beyond the scope of this paper. Alternatively, combining the results of \cite{DW} with Theorem \ref{infinitytheorem} we will be content with proving the existence of infinitely many $c\in\mathbb{R}$ for which the Michelson system generates either a heteroclinic knot or a homoclinic noose. In more detail, we prove:
\begin{theorem}
  \label{noose}  There exist countably many $c>0$ at which the Michelson system generates either a homoclinic noose or a heteroclinic knot.
\end{theorem}
\begin{proof}
To begin, we denote by $W_\pm$ the respective one-dimensional invariant manifolds of the saddle foci $p_+$ and $p_-$. In addition, let us further recall the results of \cite{DW}, where the following was proven:    
\begin{theorem}
        There exist two sequences $\{c_n\}_n$ and $\{d_n\}_n$ of positive real numbers such that the following holds:
        \begin{enumerate}
            \item    At every $c_n$ the Michelson system generates a bounded heteroclinic trajectory connecting $p_+$ and $p_-$.
            \item At every $d_n$ the Michelson system generates two homoclinic trajectories - one at $p_+$ and another at $p_-$.
        \end{enumerate}
       \end{theorem}
For a proof, see Theorem 1.3 and Theorem 1.6 in \cite{DW}. By this Theorem it follows that to prove Theorem \ref{noose} it would suffice to show that for all $c>0$ the Michelson system generates two unbounded components - one in $W_+$ and $W_-$ each - which we do by applying Theorem \ref{infinitytheorem}. To this end, we begin by proving the Michelson system satisfies all the assumptions of Theorem \ref{infinitytheorem} (see Fig.\ref{michelsonconf}). At this point we state that by Remark \ref{nocon} we already know that if the Michelson system satisfies the assumptions of Theorem \ref{infinitytheorem}, both $p_+$ and $p_-$ generate an unbounded invariant manifold.\\

\begin{figure}[h]
\centering
\begin{overpic}[width=0.5\textwidth]{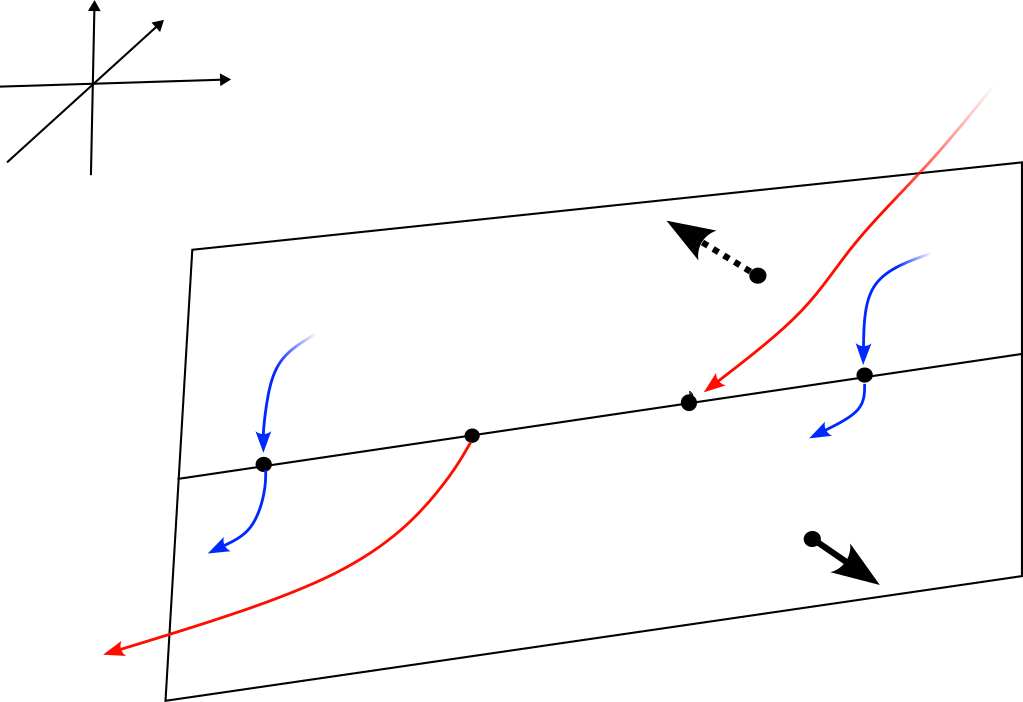}
\put(450,530){$\{\dot{x}>0\}$}
\put(430,280){$p_-$}
\put(1030,350){$\{\dot{x}<0\}$}
\put(670,250){$p_+$}
\put(500,400){$H_+$}
\put(600,120){$H_-$}
\put(230,595){$x$}
\put(165,680){$y$}
\put(75,700){$z$}
\end{overpic}
\caption{\textit{The configuration of $H=H_+\cup H_-$ and $l$ for the Michelson system (along with the directions of the vector field on them). The red curves denote the unbounded invariant manifolds given by Theorem \ref{infinitytheorem}, while the blue flow lines denote the (tangent) local dynamics on $l$.}}\label{michelsonconf}

\end{figure}

To begin, let us define $H=\{\dot{x}=0\}=\{(x,0,z)\mid x,z\in\mathbb{R}\}$ (where the velocity $\dot{x}$ is taken with respect to Eq.\ref{Mich}). Applying similar methods to those used in the previous section, we see the set $l$, the tangency set of the vector field to $H$ is parameterized by $\{(x,0,0)\mid x\in\mathbb{R}\}$ - and that in addition, $H\setminus l$ is composed of two half-planes, $H_+=\{(x,0,z)\mid z>0\}$ and $H_-=\{(x,0,z)\mid z<0\}$ such that on $H_+$ trajectories cross from $\{\dot{x}\leq0\}=\{(x,y,z)\mid y\leq0\}$ to $\{\dot{x}>0\}=\{(x,y,z)\mid y>0\}$, while in $H_-$ the opposite occurs (see Fig.\ref{michelsonconf}). Moreover, it is also easy to verify the following (see Fig.\ref{michelsonconf}):

\begin{itemize}
    \item For all $c\in\mathbb{R}$, the plane $H_r=\{(r,y,z)\mid y,z\in\mathbb{R}\}$ is transverse to $H$.
    \item Denote the vector field from Eq.\ref{Mich} by $F_c$. For all $\mid x\mid >c\sqrt{2}$, $F(x,0,0)=(0,0,c^2-\frac{x^2}{2})$ points in the negative $z$--direction. 
\end{itemize}

This implies that given $s=(x,0,0)$, $\mid x\mid >c\sqrt{2}$ the trajectory of $s$ arrives at it from $\{\dot{x}<0\}$ and re-enters $\{\dot{x}<0\}$ immediately upon leaving $s$ - as illustrated in Fig.\ref{michelsonconf}. Similarly to the previous section, it follows that in order to apply Theorem \ref{infinitytheorem} all that remains to be shown is that the Poincaré index of the Michelson system at $\infty$ is either $0,1$ or $-1$. To do so, we first note that by Proposition \ref{polyinf}, the vector field given by Eq.\ref{Mich} extends to $\infty$ (after modification) - where $\infty$ is added as a fixed point for the flow. In addition, we further note the vector field $F_c$ can point at the $(0,0,\lambda)$, $\lambda>0$ direction only on vectors which lie at the intersection $\{\dot{x}=0\}\cap\{\dot{y}=0\}$.\\

By computation we have $\{\dot{x}=0\}\cap\{\dot{y}=0\}=l$, i.e., the intersection corresponds to the straight line $\{(x,0,0)\mid x\in\mathbb{R}\}$. Now, recall that as shown earlier, for all sufficiently large $\mid x\mid $, $F(0,0,x)$ points in the negative $z$--direction - therefore, setting $S_r=\{\mid \mid s\mid \mid =r\},r>0$ a similar argument to the one used to study the Belousov--Zhabotinsky reaction proves that whenever $r>0$ is sufficiently large the map $\frac{F_c}{\mid \mid F_c\mid \mid }:S_r\to S^2$ does not include $(0,0,1)$ in its image. Or, in other words, the degree of $\frac{F_c}{\mid \mid F_c\mid \mid }$ on $S_r$ is $0$ - since $r$ can be taken to be arbitrarily large, by Definition \ref{index} we conclude the Poincaré index at $\infty$ is $0$. All in all, this shows Eq.\ref{Mich} satisfy the assumptions of Theorem \ref{infinitytheorem}, which implies both $p_+$ and $p_-$ generate a respective unbounded, invariant manifold $\Gamma_\pm$ connecting it to $\infty$. Therefore, to conclude the proof it remains to show $\Gamma_i\subseteq W_i$, where $i=+,-$.\\

We do so by showing the two-dimensional invariant manifolds of the saddle foci $p_+$ and $p_-$ are both transverse to $H$ - i.e., we prove the trajectory of any initial condition on these two invariant manifolds spirals between the regions $\{\dot{x}>0\}$ and $\{\dot{x}<0\}$ infinitely many times. Since $H=\{(x,0,z)\mid x,z\in\mathbb{R}\}$ is the vanishing set of the $\dot{x}$ velocity for Eq.\ref{Mich} and because the proof of Theorem \ref{infinitytheorem} shows the $\dot{x}$ velocity on both $\Gamma_+$ and $\Gamma_-$ does not change its sign, this would suffice to complete the proof. Again, due to the symmetric nature of Eq.\ref{Mich} it would suffice to prove this only for $p_+=(c\sqrt{2},0,0)$. To this end, we recall the Jacobian matrix at $p_+$, $J_+$, is given by the following matrix:

\begin{equation*}
    \begin{pmatrix}

0 & 1 & 0\\
0 & 0 & 1 \\
-c\sqrt{2} & -1 & 0

\end{pmatrix} 
\end{equation*}

Given a vector $(\gamma_1,0,\gamma_2)\in H$ we have the equalities $J_+(\gamma_1,0,\gamma_2)=(0,\gamma_2,-c\sqrt{2}\gamma_1)$ - which proves that for the vector $(\gamma_1,0,\gamma_2)$ to be a vector in $H=\{(x,0,z)\mid x,z\in\mathbb{R}\}$ it must satisfy $\gamma_2=0$. This immediately implies $J_+ H\ne H$, i.e., $J_+H$ is a plane transverse to $H$ - or in other words, $H=\{\dot{x}=0\}$ is not tangent to the two-dimensional invariant manifold of $p_+$, and has to be transverse to it. This proves the only possibility is $\Gamma_+\subseteq W_+$ and Theorem \ref{noose} now follows.
\end{proof}
\section{Discussion - weakening the assumptions of Theorem \ref{infinitytheorem}}
\label{nonstand}
Before concluding this paper, in this section we discuss how Theorem \ref{infinitytheorem} can possibly be generalized. Of course, one can derive more than one generalization to Theorem \ref{infinitytheorem}, as several of its assumptions can be relaxed - for example, it is easy to think of topological scenarios where, say, the cross-section $H$ is disconnected (see Fig.\ref{discc}), the Poincaré index at $\infty$ is not $0,1$ or $-1$, scenarios in which the existence of infinite number of fixed points for the flow does not present an obstruction, and so on.\\

These examples show that in general one probably should not expect a more general form of Theorem \ref{infinitytheorem}, as there can be many different generalizations of it - all depending on the topological scenario in question. Therefore, instead, in this section we exemplify how one can apply the ideas from the proof of Theorem \ref{infinitytheorem} to Dynamical Systems which do not satisfy all the assumptions of Theorem \ref{infinitytheorem}. In the spirit of Sections \ref{ap1} and \ref{ap2}, we will do so via analyzing a concrete example.\\

\begin{figure}[h]
\centering
\begin{overpic}[width=0.5\textwidth]{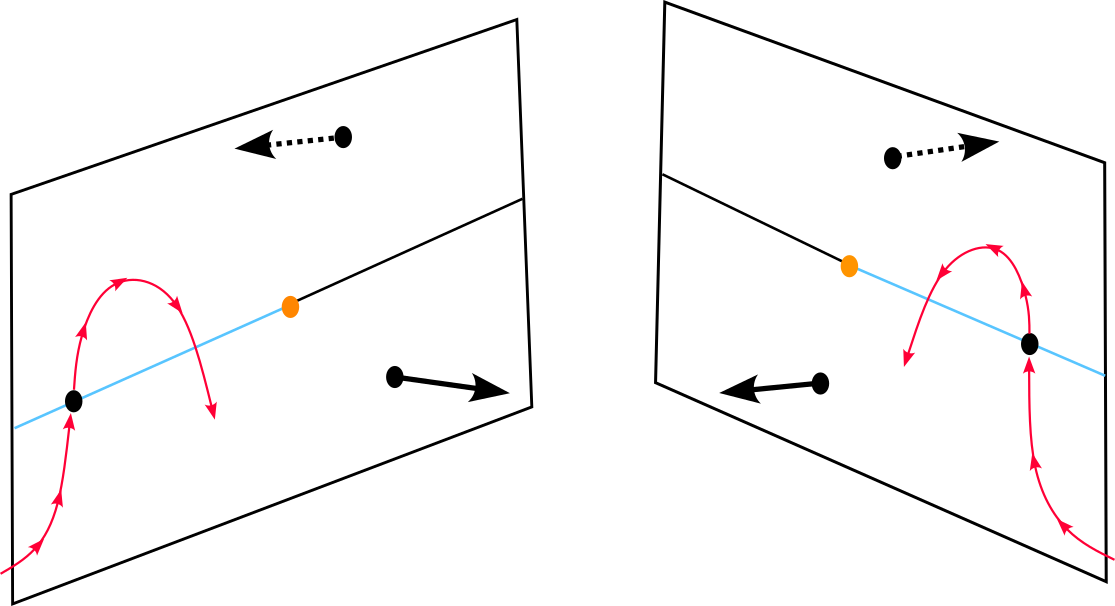}

\end{overpic}
\caption{\textit{A scenario in which the set $H$ is composed of two planes, with the fixed points $x_1$ and $x_2$ and the arcs $l_1$ and $l_2$ drawn as the orange dots and cyan lines (respectively). Based on the local dynamics on $l_1$ and $l_2$  one easily sees this Dynamical system satisfies the conclusions of Theorem \ref{infinitytheorem}.}}\label{discc}

\end{figure}

To introduce the said example, given $a>0$ consider the following vector field, originally derived from the Sprott E system and introduced at \cite{Mor}:

\begin{equation} \label{Mpr1}
\begin{cases}
\dot{x} = yz+a \\
 \dot{y} = x^2-y\\
 \dot{z}=1-4x
\end{cases}
\end{equation}
This system has precisely one fixed point, $p_a=(\frac{1}{4},\frac{1}{16},-16a)$, of Poincaré index $-1$ - for the details, see Proposition 1 in \cite{Mor}. By direct computation one can see that for $a=1$ the fixed point $p_a$ is a sink with two complex eigenvalues, of degree $-1$ - and indeed, from now on until the end of this section we will assume $a>0$ is such that  the fixed point $p_a$ has a pair of complex-conjugate eigenvalues, and we will always denote the corresponding vector field by $F_a$. Using highly similar ideas to those used to prove Theorem \ref{infinitytheorem} we now show that for any such an $a$ the one-dimensional invariant manifold of $p_a$, $W_a$, is composed of two unbounded, one-dimensional invariant manifolds connecting $p_a$ to $\infty$ - which we do despite the fact that, as we will soon see, the vector field $F_a$ does not satisfy all the assumptions of Theorem \ref{infinitytheorem}.\\ 

\begin{figure}[h]
\centering
\begin{overpic}[width=0.5\textwidth]{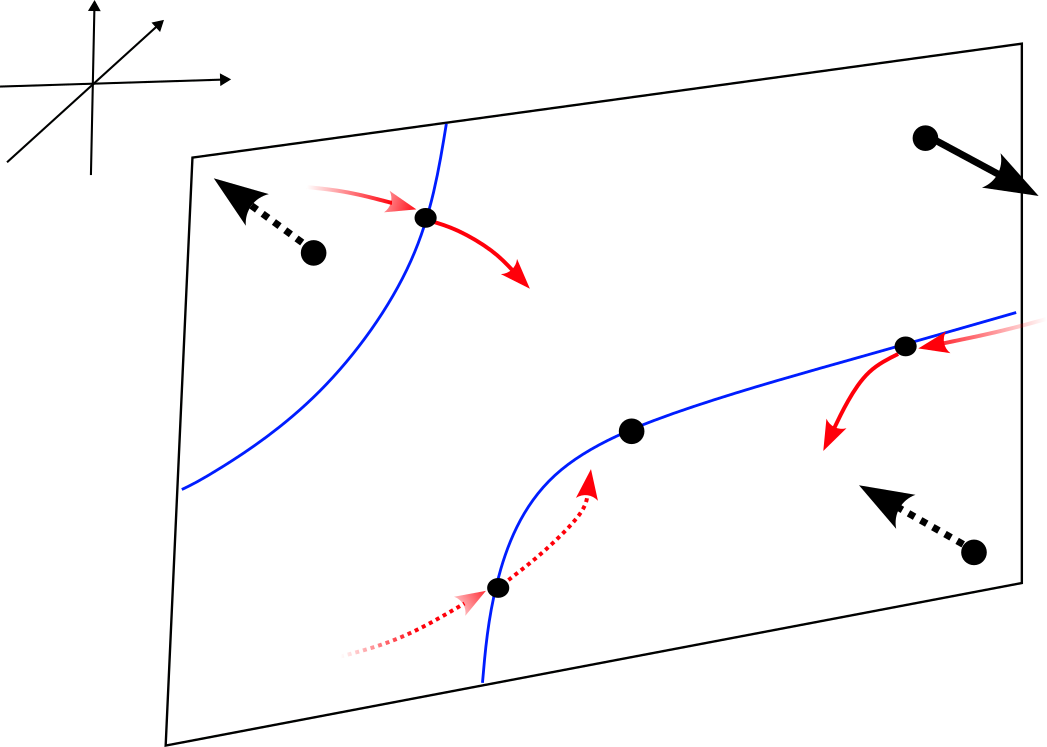}
\put(400,670){$\{\dot{z}>0\}$}
\put(440,210){$l_2$}
\put(1030,400){$\{\dot{z}<0\}$}
\put(720,300){$l_1$}
\put(570,340){$p_a$}
\put(720,450){$A_1$}
\put(700,160){$A_2$}
\put(200,400){$A_3$}
\put(380,400){$L$}
\put(225,630){$x$}
\put(160,700){$y$}
\put(75,720){$z$}
\end{overpic}
\caption{\textit{The plane $H=A_1\cup A_2\cup A_3\cup l$ and the direction of the vector field on it (where $l=L\cup l_1\cup l_2$). The flow lines tangent to $H$ at $l$ are sketched in red.}}\label{nons}

\end{figure}

To begin, we first note that since $p_a$ is isolated the maps $\frac{F_a}{\mid \mid F_a\mid \mid }:S_r\to S^2$, where again $S_r=\{s\in\mathbb{R}^3\mid \mid \mid s-p_a\mid \mid =r\}$ are homotopic as we vary $r>0$. Therefore, since for sufficiently small $r$ the degree of $\frac{F_a}{\mid \mid F_a\mid \mid }$ is $-1$. the same is true for all $r>0$. Hence, the Poincaré index at $\infty$ is $1$. Now, choose the velocity vanishing set $H=\{\dot{z}=0\}=\{(\frac{1}{4},y,z)\mid y,z\in\mathbb{R}\}$. This plane is transverse to all horizontal planes given by $H_c=\{(x,y,c)\mid x,y\in\mathbb{R}\}$, $c\in\mathbb{R}$, and that we have $\{\dot{z}>0\}=\{(x,y,z)\mid x<\frac{1}{4}\}$, $\{\dot{z}<0\}=\{(x,y,z)\mid x>\frac{1}{4}\}$ (see Fig.\ref{nons}). In addition, setting $J_a$ as the Jacobian matrix at $p_a$, using similar ideas to those presented in the end of the proof of Theorem \ref{noose} if follows $J_a H$ is a plane transverse to $H$ - which, similarly, implies the two-dimensional invariant manifold of $p_a$ is transverse to $H$. Continuing our analysis of the local dynamics on $H$, as the normal vector to $H$ is $(1,0,0)$ it similarly follows (again, by direct computation) that the tangency set of the vector field to $H$ is given by the curve $l=\{(\frac{1}{4},y,\frac{-a}{y})\mid y\in\mathbb{R}\}$.\\

Unlike the dynamical systems considered in Sections \ref{ap1} and \ref{ap2}, this time the curve $l$ does not satisfy the assumptions of Theorem \ref{infinitytheorem} - if only because $l$ is not homeomorphic to a real line (see Fig.\ref{nons}). In particular, it follows $H\setminus l$ is composed of three regions - $A_1,A_2$ and $A_3$, as illustrated in Fig.\ref{nons} and Fig.\ref{nons2}. However, as we now prove, certain elements in the proof of Theorem \ref{infinitytheorem} still apply - in particular, we show that despite the problematic topological structure of $l$, the properties of the vector field still allow us to apply many of the ideas presented in the proof of Theorem \ref{infinitytheorem}.\\

\begin{figure}[h]
\centering
\begin{overpic}[width=0.5\textwidth]{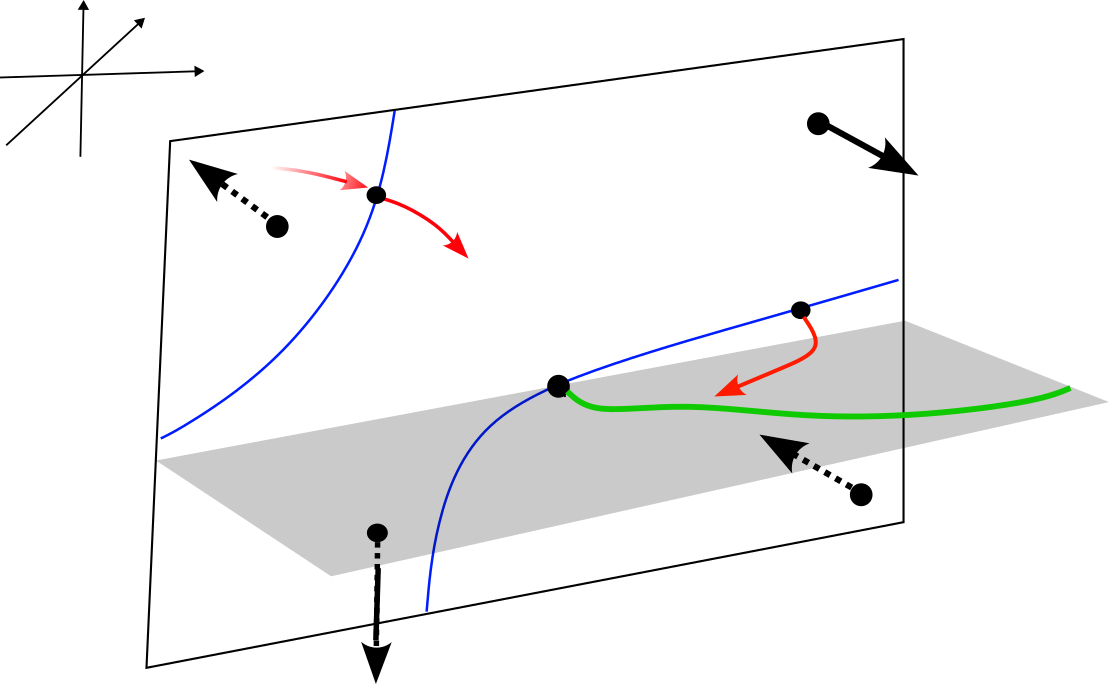}
\put(400,590){$\{\dot{z}>0\}$}
\put(440,150){$H'$}
\put(1000,350){$\{\dot{z}<0\}$}
\put(720,360){$l_1$}
\put(490,300){$p_a$}
\put(620,450){$A_1$}
\put(700,150){$A_2$}
\put(200,350){$A_3$}
\put(320,350){$L$}
\put(185,545){$x$}
\put(135,610){$y$}
\put(65,625){$z$}
\end{overpic}
\caption{\textit{The plane $H$ and its intersection with the (gray) half-plane $H'=H_0\cap\{\dot{z}<0\}$. In this scenario the trajectories of initial conditions $s\in l_1$ flow to the green curve.}}\label{nons2}

\end{figure}

We begin by further studying the local dynamics on $H$. Recalling the planar domains $A_1,A_2$ and $A_3$ defined above (see Figs.\ref{nons} and \ref{nons2}), by direct computation it follows that whenever $s\in A_1$ we have $F(s)\bullet(1,0,0)<0$, while for $s\in A_2,A_3$ we have $F(s)\bullet(1,0,0)>0$. In other words, $A_1$ is the set where trajectories cross from $\{\dot{z}>0\}$ to $\{\dot{z}<0\}$, while in $A_2$ and $A_3$ the opposite occurs (see Fig.\ref{nons2}). Now, note that for all $y\in\mathbb{R}$ we have $F(\frac{1}{4},y,-\frac{a}{y})=(0,\frac{1}{16}-y,0)$ - as $l=\{(\frac{1}{4},y,-\frac{a}{y})$, writing $s=(\frac{1}{4},y,-\frac{a}{y})$ and setting $L=\{(\frac{1}{4},y,-\frac{a}{y})\mid y<0\}$, $l_2=\{(\frac{1}{4},y,-\frac{a}{y})\mid \frac{1}{16}>y>0\}$, $l_1=\{(\frac{1}{4},y,-\frac{a}{y})\mid y>\frac{1}{16}\}$ we conclude that for $s\in l$ we have precisely one of the following (see the illustration in Fig.\ref{nons}):

\begin{enumerate}
    \item For $s\in L$ we have $y<0$ which implies $F(s)$ points in the positive $y$ direction - hence the flow lines arrive at $s$ from $\{\dot{z}<0\}$ and returns to $\{\dot{z}<0\}$ immediately upon leaving $s$.
    \item For $s\in l_1$ corresponding to $0<y<\frac{1}{16}$ again $F(s)$ points to the positive $y$--direction - but this time the same arguments imply the flow line arrives at $s$ from $\{\dot{z}>0\}$ and re-enters $\{\dot{z}>0\}$ immediately upon leaving $s$.
    \item Finally, for $s\in l_2$ corresponding to $y>\frac{1}{16}$ the vector $F(s)$ points in the negative $y$--direction. This implies the flow lines arrive at $s$ from $\{\dot{z}<0\}$ and return to it immediately upon leaving $s$.
\end{enumerate}

We now consider the plane $H_0=\{(x,y,-16a)\mid x,y\in\mathbb{R}\}$ (see Fig.\ref{nons2}) - by using similar ideas to those applied in the proof of Theorem \ref{infinitytheorem}, we now prove the existence of an unbounded invariant manifold for $p_a$. We begin by considering the trajectory of some initial conditions $s\in l_2$. There are precisely two possibilities (see Fig.\ref{nons2}):

\begin{enumerate}
    \item The trajectory of $s$ remains trapped forever in $\{\dot{z}<0\}$. Since the $z$--component of $s$ is greater than $-16a$ it follows that in this case the trajectory of $s$ eventually hits $H_0\cap\{\dot{z}<0\}$ transversely and enters $\{(x,y,z)\mid z<-16a\}$.
    \item The trajectory of $s$ eventually leaves $\{\dot{z}<0\}$, i.e., it hits transversely $A_3\cup A_3$ and enters $\{\dot{z}<0\}$. We now show the said trajectory must hit $A_2$ before it can hit $A_3$ - to this end, consider the half-plane $H_1=\{(x,\frac{1}{6},z)\mid x,z\in\mathbb{R}\}$. As the normal vector to $H_1$ is $(0,1,0)$ it is easy to prove that on $H_1$ separates $l_1$ and $A_3$ inside the half-space $\{\dot{z}\leq0\}=\{(x,y,z)\mid x\geq\frac{1}{4}\}$ it follows the trajectories of initial conditions on $l_1$ cannot hit $A_3$ \textbf{before} hitting $A_2$ - or in other words, in order for the trajectory of $s$ to hit $A_3$ (or $L$) it must first hit $A_2$ transversely and enter $\{\dot{z}<0\}=\{(x,y,z)\mid x<\frac{1}{4}\}$.
\end{enumerate}

All in all, we conclude the existence of a two-dimensional set $V$ made of the flow-lines connecting the initial conditions $s\in l_1$ to $A_2\cup H_0$ (see Fig.\ref{nons2}). Similarly to the arguments used in the proof of Theorem \ref{infinitytheorem} we conclude that under similar idealized assumptions on the behavior of the vector field at $\infty$, the two-dimensional set $V\cup H_0\cup A_2$ traps a topological cone $C_1$ with a tip at the fixed-point $p_a$. In particular, given any $s\in\partial C_1$, $F_a(s)$ satisfies precisely one of the following:

\begin{itemize}
    \item If $s\in V$, then $F_a(s)$ is tangent to $V$.
    \item If $s\in A_2$, then $F_a(s)$ points into $\{\dot{z}>0\}$.
    \item If $s\in H_0$, then $F_a(s)$ points into $\{(x,y,z)\mid z<-16a\}$.
\end{itemize}

Or in other words, on every $s\in\partial C_1$ the vector field either points outside of $C_1$ or lies tangent to $\partial C_1$ - which yields no trajectory can enter $C_1$ under the flow. Consequentially, similar arguments to those used to prove Theorem \ref{infinitytheorem} imply $p_a$ generates some unbounded invariant manifold $\Gamma_1\subseteq\{(x,y,z)\mid z>-16a\}\cap\{\dot{z}<0\}$ (see Fig.\ref{nons4}).\\

\begin{figure}[h]
\centering
\begin{overpic}[width=0.5\textwidth]{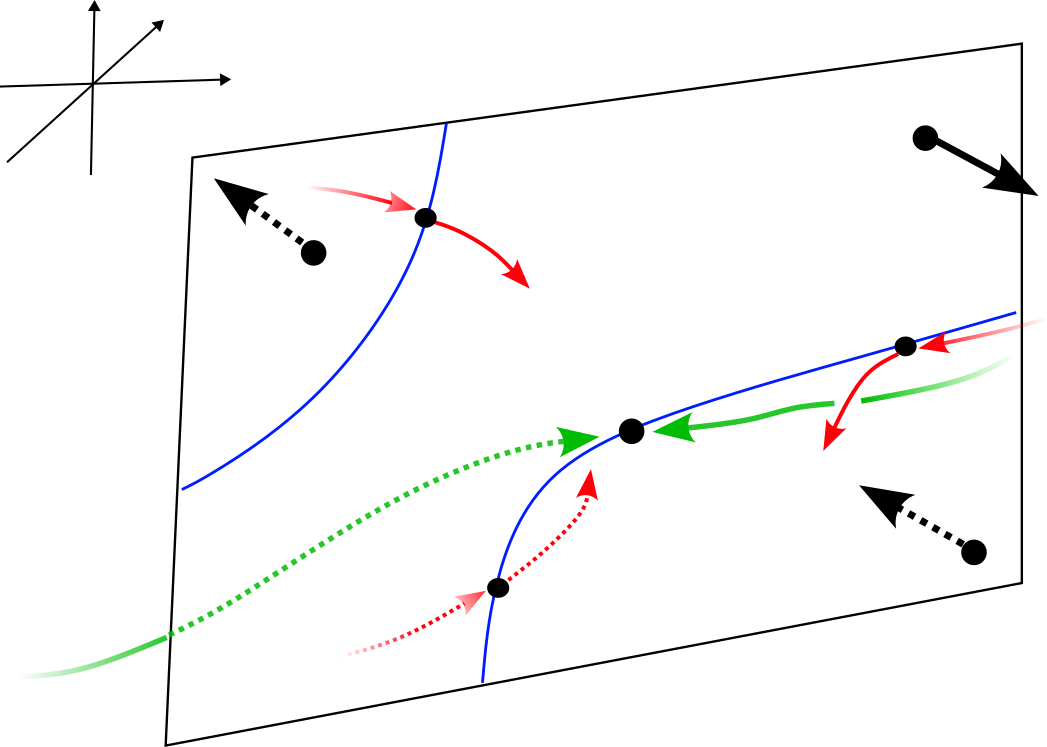}
\put(400,670){$\{\dot{z}>0\}$}
\put(430,215){$\Gamma_2$}
\put(1030,400){$\{\dot{z}<0\}$}
\put(720,270){$\Gamma_1$}
\put(570,340){$p_a$}
\put(720,450){$A_1$}
\put(700,160){$A_2$}
\put(200,400){$A_3$}
\put(380,400){$L$}
\put(225,630){$x$}
\put(160,700){$y$}
\put(75,720){$z$}
\end{overpic}
\caption{\textit{The invariant manifolds $\Gamma_1$ and $\Gamma_2$.}}\label{nons4}

\end{figure}

We now sketch the proof of the analogous result for $l_2=\{(\frac{1}{4},y,-\frac{a}{y})\mid 0<y<\frac{1}{16}\}$. As the $z$--coordinate of $s$ is less than $-16a$, by using similar arguments to those above now yield the trajectory of every initial condition $s\in l_2$ satisfies the following:
\begin{enumerate}
    \item The trajectory of $s$ remains trapped inside $\{\dot{z}>0\}$ - in which case it hits $H_0\cap\{\dot{z}>0\}$ transversely and enters $\{(x,y,z)\mid z>-16a\}$. In particular, $H_0\cap\{\dot{z}\geq0\}$ separates $l_2$ from $A_3$.
    \item  The trajectory of $s$ eventually escapes $\{\dot{z}\geq0\}$ into $\{\dot{z}<0\}$ by hitting $A_1\cap\{(x,y,z)\mid z\leq-16a\}$ transversely.
\end{enumerate}

Using similar arguments, we conclude the existence of $S$, a two-dimensional set of flow lines connecting $l_2$ and $H_0\cup A_1$. Similarly, this also yields the existence of an invariant manifold in $\Gamma_2\subseteq\{(x,y,z)\mid z<-16a\}\cap\{\dot{z}>0\}$ connecting $p_a$ to $\infty$, as illustrated in Fig.\ref{nons4}. We may now summarize our results as follows:
\begin{theorem}
    For all $a\in\mathbb{R}$, $a\ne0$ such that the fixed point $p_a$ has a pair of complex-conjugate eigenvalues, the system \ref{Mpr1} has two invariant manifolds, $\Gamma_1$ and $\Gamma_2$ connecting $p_a$ to $\infty$. Moreover,  $\Gamma_1\cup\Gamma_2$ forms the one-dimensional invariant manifold of $p_a$ and the union $\{p_a,\infty\}\cup\Gamma_1\cup\Gamma_2$ is a curve in $S^3$ ambient isotopic to $S^1$.
\end{theorem}
\begin{proof}
    Using similar arguments to those used in the proof of Theorem \ref{infinitytheorem} it is easy to prove $\Gamma_1$ and $\Gamma_2$ are not knotted or linked with one another, therefore $\Gamma_1\cup\Gamma_2\cup\{p_a,\infty\}$ is a knot in $S^3$ ambient isotopic to $S^1$. As such, we need only prove $\Gamma_1\cup\Gamma_2=W_1$, where $W_1$ is the one-dimensional invariant manifold of $p_a$. To do so, we recall that as proven in the beginning of this section, the two-dimensional invariant manifold of $p_a$, $W_2$, is transverse to the plane $H$ - which, since $p_a$ has two complex-conjugate eigenvalues and by $H=\{\dot{z}=0\}$ implies the $\dot{z}$ velocity of along any flow line on $W_2$ vanishes infinitely many times. Since $\Gamma_1\subseteq\{\dot{z}<0\}$, $\Gamma_2\subseteq\{\dot{z}>0\}$ the sign of $\dot{z}$ is constant on both $\Gamma_1$ and $\Gamma_2$, hence they can only be a part of $W_1$. All in all, $W_1=\Gamma_1\cup\Gamma_2$ and the assertion follows . 
\end{proof}
\section{Data availability statement}
No data was used for this paper.
\section{Competing interests statement}
The author states that there is no conflict of interest, and that he did not receive support from any organization for the submitted work. Moreover, the authors has no financial or proprietary interests in any material discussed in this article.
\section{Funding statement}
The author did not receive support or funding from any organization for the submitted work.
\section{Author contribution statement}
The sole author confirms sole responsibility for all aspects of the work. The author contributed to the study conception and design, material preparation, analysis, preparing all figures, and manuscript writing. The author has read and approved the final manuscript.
\printbibliography

@article{S,
        author = "Smale, S.",
        title = "Differentiable dynamical systems",
          journal = "Bull. Amer. Math. Soc.",
        volume = "73",
        year = "1967",
        pages = "747-817",
}

@article{LeS,
        author = "Shilnikov, L.",
        title = "A case of the existence of a denumerable set of periodic motions",
          journal = "Sov. Math. Dok.",
        volume = "6",
        year = "1967",
        pages = "163-166",
}

@article{LiLl,
        author = "Lima, M.F.S., and Llibre, J.",
        title = "Global dynamics of the Rössler system with
conserved quantities",
          journal = "J. Phys. A: Math. Theor. ",
        volume = "44",
        year = "2011"
}

@article{K,
        author = "Krischenko, A.P.",
        title = "Estimations of domains with cycles",
          journal = "Computers and Mathematics with Applications",
        volume = "34",
        year = "1997",
        pages = "325-332"
}

@article{HYCX,
        author = "Chen, H., and Liu, Y., and Feng, C., and Liu, A., and Huang, X.",
        title = "Dynamics at Infinity and Existence of Singularly Degenerate Heteroclinic Cycles in Maxwell–Bloch System",
          journal = "Journal of Computational and Nonlinear Dynamics",
        volume = "15",
        year = "2020",
}

@article{hyper,
        author = "Blumension, L.E.",
        title = "A Derivation of n-Dimensional Spherical Coordinates",
          journal = "The American Mathematical Monthly",
        volume = "67 (1)",
        year = "1960",
}

@article{Mes,
        author = "Messias, M.",
        title = "Dynamics at infinity of a cubic Chua's system",
          journal = "International Journal of Bifurcation and Chaos",
        volume = "21 (1)",
        year = "2011",
}

@article{Yun,
        author = "Liu, Y.",
        title = "Dynamics at infinity and the existence of singularly degenerate heteroclinic cycles in the conjugate Lorenz-type system",
          journal = "Nonlinear Analysis: Real World Applications",
        volume = "13",
        year = "2012",
}

@article{Mor,
        author = "Wang, Z., and Moroz, I., and Wei, Z., and H., Ren",
        title = "Dynamics at infinity and a Hopf bifurcation arising in a quadratic system with coexisting attractors",
          journal = "Pramana – Journal of Physics",
        volume = "90",
        year = "2018",
}

@article{DW,
        author = "Wilczak, D.",
        title = "The Existence of Shilnikov Homoclinic Orbits in the Michelson System: A Computer-Assisted Proof",
          journal = "Foundations of Computational Mathematics",
        volume = "6",
        year = "2006",
}

@article{Michh,
        author = "Michelson, D.",
        title = "Steady solutions of the Kuramoto–Sivashinsky equation",
          journal = "Physica D",
        volume = "19",
        year = "1986",
}

@book{Perko,
  author = "Perko, L.",
  year = "2001",
  title = "Differential Equations and Dynamical Systems, Third Edition",
  publisher = "Springer"
}

@book{Mil,
  author = "John W. Milnor",
  year = "2001",
  title = "Topology from the Differentiable viewpoint",
  publisher = "New Jersey: World Scientific"
}

@article{BZ,
        author = "Argoul, F., and Arneodo, A., and Richetti, P.",
        title = "Experimental evidence for homoclinic chaos in the Belousov-Zhabotinskii reaction",
          journal = "Physics Letter A",
        volume = "120",
        year = "1987"
}

@article{GT,
        author = "Genesio, R. and Tesi, A.",
        title = "Harmonic balance methods for the analysis of chaotic dynamics in nonlinear systems",
          journal = "Automatica",
        volume = "28",
        year = "1992"
}

@article{I,
    author = "Igra, E.",
    title = "Knots and Chaos in the Rössler system",
    journal  = "Journal of Differential Equations",
    volume = "437",
    year = "2025"
}

@article{Pi,
    author = "Pinsky, T.",
    title ="Analytical study of the Lorenz system: Existence of infinitely many periodic orbits and their topological characterization
" ,
    journal ="Proceedings of the National Academy of Sciences" ,
    volume = "120",
    year ="2023"
}
\end{document}